\documentclass[12pt,english]{article}
\usepackage[T1]{fontenc}
\usepackage[latin9]{inputenc}
\usepackage{geometry}
\usepackage{amsmath}
\usepackage{amsthm}
\usepackage{amssymb}

\makeatletter
\numberwithin{equation}{section}
\theoremstyle{definition}
\newtheorem{defn}{\protect\definitionname}[section]
\theoremstyle{definition}
\newtheorem{example}{\protect\examplename}[section]
\theoremstyle{plain}
\newtheorem{thm}{\protect\theoremname}[section]
\theoremstyle{remark}
\newtheorem{rem}{\protect\remarkname}[section]
\theoremstyle{plain}
\newtheorem{prop}{\protect\propositionname}[section]
\theoremstyle{plain}
\newtheorem{lem}{\protect\lemmaname}[section]
\theoremstyle{plain}
\newtheorem{cor}{\protect\corollaryname}[section]
\theoremstyle{plain}
\newtheorem{question}{\protect\questionname}

\usepackage{graphicx}
\usepackage{hyperref}
\usepackage{comment}
\usepackage{amsmath}
\usepackage{cases}
\usepackage[title]{appendix}
\usepackage{float}

\hypersetup{
     colorlinks   = true,
     citecolor    = blue,
     linkcolor    = blue
}

\date{}

\makeatother

\usepackage{babel}
\providecommand{\corollaryname}{Corollary}
\providecommand{\definitionname}{Definition}
\providecommand{\examplename}{Example}
\providecommand{\lemmaname}{Lemma}
\providecommand{\propositionname}{Proposition}
\providecommand{\questionname}{Question}
\providecommand{\remarkname}{Remark}
\providecommand{\theoremname}{Theorem}

\begin{document}
\title{${\rm SL}_{2}(\mathbb{R})$-developments and Signature Asymptotics
for Planar Paths with Bounded Variation}
\author{H. Boedihardjo\thanks{Department of Statistics, University of Warwick, Coventry, CV4 7AL,
United Kingdom. Email: horatio.boedihardjo@warwick.ac.uk. }$\ $ and X. Geng\thanks{School of Mathematics and Statistics, University of Melbourne, Parkville
VIC 3010, Australia. Email: xi.geng@unimelb.edu.au. XG gratefully
acknowledges the support of ARC grant DE210101352.}}
\maketitle
\begin{abstract}
The signature transform, defined by the formal tensor series of global
iterated path integrals, is a homomorphism between the path space
and the tensor algebra that has been studied in geometry, control
theory, number theory as well as stochastic analysis. An elegant isometry
conjecture states that the length of a bounded variation path $\gamma$
can be recovered from the asymptotics of its normalised signature:
\[
\text{Length}(\gamma)=\lim_{n\rightarrow\infty}\big\Vert n!\int_{0<t_{1}<\cdots<t_{n}<T}d\gamma_{t_{1}}\otimes\cdots\otimes d\gamma_{t_{n}}\big\Vert^{\frac{1}{n}}.
\]
This property depends on a key topological non-degeneracy notion known
as tree-reducedness (namely, with no tree-like pieces). Existing arguments
have relied crucially on $\gamma$ having a continuous derivative
under the unit speed parametrisation. In this article, we prove the
above isometry conjecture for planar paths by assuming only local
bounds on the angle of $\gamma'$ (which ensures the absence of tree-like
pieces). Our technique is based on lifting the path onto the special
linear group ${\rm SL}_{2}(\mathbb{R})$ and analysing the behaviour
of the associated angle dynamics at a microscopic level.
\end{abstract}
\newpage

\tableofcontents{}

\section{Introduction}

The \textit{signature transform} (or simply the \textit{signature})
of a multidimensional path $\gamma:[0,L]\rightarrow\mathbb{R}^{d}$
is the formal tensor series 
\[
S(\gamma)\triangleq\sum_{n=0}^{\infty}\int_{0<t_{1}<\cdots<t_{n}<L}d\gamma_{t_{1}}\otimes\cdots\otimes d\gamma_{t_{n}}
\]
formed by the global iterated path integrals of all orders. Such a
transformation was originally introduced by K.T. Chen \cite{Chen73}
to construct a cohomology theory on loop spaces over manifolds, which
had led to far-reaching applications in geometry and algebraic topology.
It also played an essential role in Dyson's quantum field theory (cf.
\cite{Dyson49}). Due to the vast development of analytic techniques
in the rough path theory, the study of the signature transform has
been enhanced to a new level of maturity over the last decade by many
authors. A landmark result was the uniqueness theorem proved by Hambly-Lyons
in their well-known work \cite{HL10} in 2010. The uniqueness theorem
asserts that the signature of a bounded variation path uniquely determines
the underlying path up to tree-like pieces (heuristically, a tree-like
piece is a portion of the path in which it travels out and reverses
back along itself). This result was later generalised to the rough
path context in \cite{BGLY16}. The uniqueness theorem has stimulated
a stream of exciting problems related to reconstructing paths from
their signatures and studying paths through functions on the signature
space (cf. \cite{CL19,MonotoneInversion,Geng17,LX18}). One reason
to work with the signature transform is that it has a nice intrinsic
algebraic structure (linearisation of non-linear path functions) that
is concealed at the level of paths (cf. \cite{Reu93}).

As a consequence of the uniqueness theorem, one naturally expects
that many quantitative properties of a path can be recovered from
its signature. In the bounded variation case, there is an elegant
and important question along this line. A simple application of the
triangle inequality shows that the signature of a continuous path
$\gamma$ with finite length satisfies the following estimate:
\[
\big\|\int_{0<t_{1}<\cdots<t_{n}<L}d\gamma_{t_{1}}\otimes\cdots\otimes d\gamma_{t_{n}}\big\|\leqslant\int_{0<t_{1}<\cdots<t_{n}<L}|d\gamma_{t_{1}}|\cdot\cdots\cdot|d\gamma_{t_{n}}|=\frac{{\rm Length}(\gamma)^{n}}{n!}
\]
for every $n\geqslant1.$ What is non-trivial and surprising is that,
this estimate becomes asymptotically sharp as $n\rightarrow\infty.$
It was conjectured implicitly in \cite{HL10} and later made explicit
by Chang-Lyons-Ni \cite{CLN18} that, for any continuous, \textit{tree-reduced}
(i.e. not containing tree-like pieces) path with finite length, after
normalisation one expects that

\begin{equation}
{\rm Length}(\gamma)=\lim_{n\rightarrow\infty}\big\| n!\big(\int_{0<t_{1}<\cdots<t_{n}<L}d\gamma_{t_{1}}\otimes\cdots\otimes d\gamma_{t_{n}}\big)\big\|^{1/n}.\label{eq:LengConj}
\end{equation}
This conjectural property is deep and surprising, as it suggests that
the tree-reduced property is eliminating the fine-scale interactions
and cancellations of the path increments in the $n$-th order iterated
integral as $n$ increases. Understanding such a phenomenon is an
important step towards obtaining effective signature lower bounds,
which is in turn essential for establishing convergence of signature
inversion schemes. This point is particularly relevant in the work
of Chang-Lyons \cite{CL19}, which is also a key missing ingredient
to theorise their proposed inversion algorithm in the general bounded
variation context.

The signature asymptotics formula (\ref{eq:LengConj}) was established
in \cite{HL10,LX15} for ${\cal C}^{1}$-paths (i.e. continuously
differentiable) parametrised at unit speed. There is an important
reason for pushing our understanding towards the general bounded variation
case. The conjectural formula (\ref{eq:LengConj}) as well as other
similar signature inversion properties, if proven to be true, should
be a pure consequence of \textit{tree-reducedness} as suggested by
the uniqueness theorem. As a result, identifying a suitable condition
to quantify the ``degree of being tree-reduced'' is an essential
step towards establishing general signature inversion properties.
If a path $\gamma$ is ${\cal C}^{1}$ \textit{at unit speed}, one
can see that $\gamma$ cannot produce a tree-like piece. Indeed, if
$\gamma'$ is a continuous function on $S^{n},$ it immediately rules
out the possibility that $\gamma$ makes an $\pi$-turn at some point.
However, this viewpoint embeds regularity assumptions of the path
into the detection of tree-reducedness, making the latter property
opaque.

For a deeper understanding about the essence of the underlying phenomenon,
it is important to develop an approach that separates the non-degeneracy
property of tree-reducedness from regularity properties of the path
and reveals a property like (\ref{eq:LengConj}) as a consequence
of tree-reducedness. The main contribution of the present article
is to provide an attempt along this philosophy. Our intuition behind
capturing tree-reducedness is very simple: we require that the path
does not make a $\pi$-turn locally, and if it does it makes it in
a way avoiding the creation of a tree-like piece (cf. Definition \ref{def:SRed}
for a more precise formulation). Our main result is stated in Theorem
\ref{thm:Mthm} below, which confirms the conjectural formula (\ref{eq:LengConj})
for planar paths with bounded variation that satisfy such a condition.
At the moment, extending the current analysis to higher dimensions
is a challenging task (cf. Section \ref{sec:FurQ} for a brief discussion).
Nonetheless, the two dimensional situation already appears to be highly
non-trivial and contains several essential ideas. Our methodology,
which is partly inspired by the series \cite{BGS20,CL16,HL10,LS06,LX15}
of works, is based on lifting the underlying path onto the special
linear group ${\rm SL}_{2}(\mathbb{R})$ and developing fine analysis
on the behaviour of the associated angle dynamics. It has a similar
nature as the method developed in \cite{BGS20}, however, the underlying
difficulties are in different directions. The method in \cite{BGS20}
deals with multi-level interactions of different signature components
(which is only relevant if the path is rough), while the current work
deals with fine-scale interactions of different time periods due to
rapid oscillations of the path. It is reasonable to expect that, suitable
combination of the two viewpoints may lead to deeper understanding
towards the more general rough path situation. We mention that ${\rm SL}_{2}(\mathbb{R})$-developments
were also used by Lyons-Sidorova \cite{LS06} to establish a decay
property of the logarithmic signature. In their work, the ${\rm SL}_{2}(\mathbb{R})$-structure
was mainly used for the geometry of its exponential map. In our approach,
the ${\rm SL}_{2}(\mathbb{R})$-structure surprisingly simplifies
the ODE dynamics by producing a decoupled ODE system for the ${\rm SL}_{2}(\mathbb{R})$-action
which may not be the case under other types of developments.

\medskip \noindent \textbf{Organisation}. In Section \ref{sec:StmMthm},
we recall a few notation and state our main theorem. In Section \ref{sec:PathDev},
we recall some basic notions on Cartan developments of paths and derive
the core equations in the context of ${\rm SL}_{2}(\mathbb{R})$-developments.
In Section \ref{sec:GloLem}, we establish several preliminary lemmas
on the angle dynamics that are critical for later analysis. In Section
\ref{sec:PfMthm}, we develop the proof of the main theorem. In Section
\ref{sec:MCusp}, we discuss how our method can be adapted to deal
with a more singular type of paths. In Section \ref{sec:FurQ}, we
discuss a few natural questions to be further investigated.

\section{\label{sec:StmMthm}Statement of the main result}

In this section, we provide the main set-up of the present article
and state our main result.

We start by recalling some standard notation about paths and their
signatures. Let $V$ be a Banach space. For each $n\geqslant1,$ we
denote $V^{\otimes n}$ as the completion of the algebraic tensor
product $V^{\otimes_{{\rm a}}n}$ under the \textit{projective tensor
norm}, which is defined by (cf. \cite{Ryan02})
\begin{equation}
\|\xi\|_{\pi}\triangleq\inf\big\{\sum_{i}|v_{1}^{i}|\cdot\cdots\cdot|v_{n}^{i}|:\xi=\sum_{i}v_{1}^{i}\otimes\cdots\otimes v_{n}^{i}\big\},\ \ \ \xi\in V^{\otimes_{{\rm a}}n}.\label{eq:ProjNorm}
\end{equation}
The projective tensor norm is the largest among all admissible tensor
norms (cf. \cite{Ryan02}).
\begin{defn}
Let $\gamma:[0,L]\rightarrow V$ be a continuous path with finite
length. The \textit{signature} of $\gamma$ is the formal tensor series
of global iterated path integrals against $\gamma$ defined by 
\[
S(\gamma)\triangleq\big(1,\gamma_{L}-\gamma_{0},\cdots,\int_{0<t_{1}<\cdots<t_{n}<L}d\gamma_{t_{1}}\otimes\cdots\otimes d\gamma_{t_{n}},\cdots\big)\in\prod_{n=0}^{\infty}V^{\otimes n}.
\]
\end{defn}
Let $\gamma$ be a given continuous path with finite length. We define
the following normalised signature asymptotics functional:
\begin{equation}
L_{1}(\gamma)\triangleq\lim_{n\rightarrow\infty}\big\| n!\big(\int_{0<t_{1}<\cdots<t_{n}<L}d\gamma_{t_{1}}\otimes\cdots\otimes d\gamma_{t_{n}}\big)\big\|_{\pi}^{1/n}.\label{eq:LimFunc}
\end{equation}
It is known that (cf. \cite{BG19,CLN18}) the limit in (\ref{eq:LimFunc})
is well defined and the quantity $L_{1}(\gamma)$ remains the same
if the limit is replaced by the supremum over $n\geqslant1$. Using
the triangle inequality, it is immediate to see that $L_{1}(\gamma)\leqslant\|\gamma\|_{{\rm 1}\text{-var}}$.
The reserve inequality (for tree-reduced paths) is the main challenging
question.

In the present article, we restrict ourselves to two dimensional paths.
We assume that $\mathbb{R}^{2}$ is equipped with the Euclidean norm
and the tensor products $(\mathbb{R}^{2})^{\otimes n}$ ($n\geqslant1$)
are equipped with the associated projective tensor norm (cf. (\ref{eq:ProjNorm})).
We consider continuous paths in $\mathbb{R}^{2}$ with finite lengths,
parametrised at unit speed. Mathematically, these paths are defined
by 
\begin{equation}
\gamma:[0,L]\rightarrow\mathbb{R}^{2},\ \gamma_{t}=(x_{t},y_{t})=\big(x_{0}+\int_{0}^{t}\cos\beta_{s}ds,y_{0}+\int_{0}^{t}\sin\beta_{s}ds\big),\label{eq:DefPath}
\end{equation}
where $\beta:[0,L]\rightarrow\mathbb{R}$ is a (Lebesgue) measurable
function. For the purpose of this paper, readers may take (\ref{eq:DefPath})
as the definition of $\gamma$. In the appendix, we will outline an
argument that every non-constant continuous path with finite $1-$variation
can be reparametrised into the form (\ref{eq:DefPath}). 

We are going to propose a natural sufficient condition that captures
the tree-reduced property, and to establish the asymptotics formula
(\ref{eq:LengConj}) for paths satisfying such condition. For the
sake of preciseness, we recapture the definition of tree-reducedness
as follows (cf. \cite{BGLY16}). Recall that, a \textit{real tree}
is a metric space in which any two distinct points can be joined by
a unique non-self-intersecting path up to reparametrisation and such
a path is a geodesic in the metric sense.
\begin{defn}
\label{def:TRed} Let $\gamma:[s,t]\rightarrow E$ be a continuous
path in some topological space $E$. We say that $\gamma$ is \textit{tree-like},
if there exists a real tree ${\cal T}$ and two continuous maps 
\[
\xi:[s,t]\rightarrow{\cal T},\ \Phi:{\cal T}\rightarrow E,
\]
such that $\xi_{s}=\xi_{t}$ and $\gamma=\Phi\circ\xi.$ A \textit{tree-like
piece} of a path $\gamma:[s,t]\rightarrow E$ is a portion $[u,v]\subseteq[s,t]$
such that $\gamma|_{[u,v]}$ is tree-like. A path is said to be \textit{tree-reduced
}if it does not contain any tree-like pieces.
\end{defn}
\noindent Heuristically, being tree-reduced means that there is no
portion of the path $\gamma$ along which it reverses back to cancel
itself right away. In the figure below, the first path is tree-reduced
while the second one contains a tree-like piece.\begin{figure}[H]  
\begin{center}  
\includegraphics[scale=0.25]{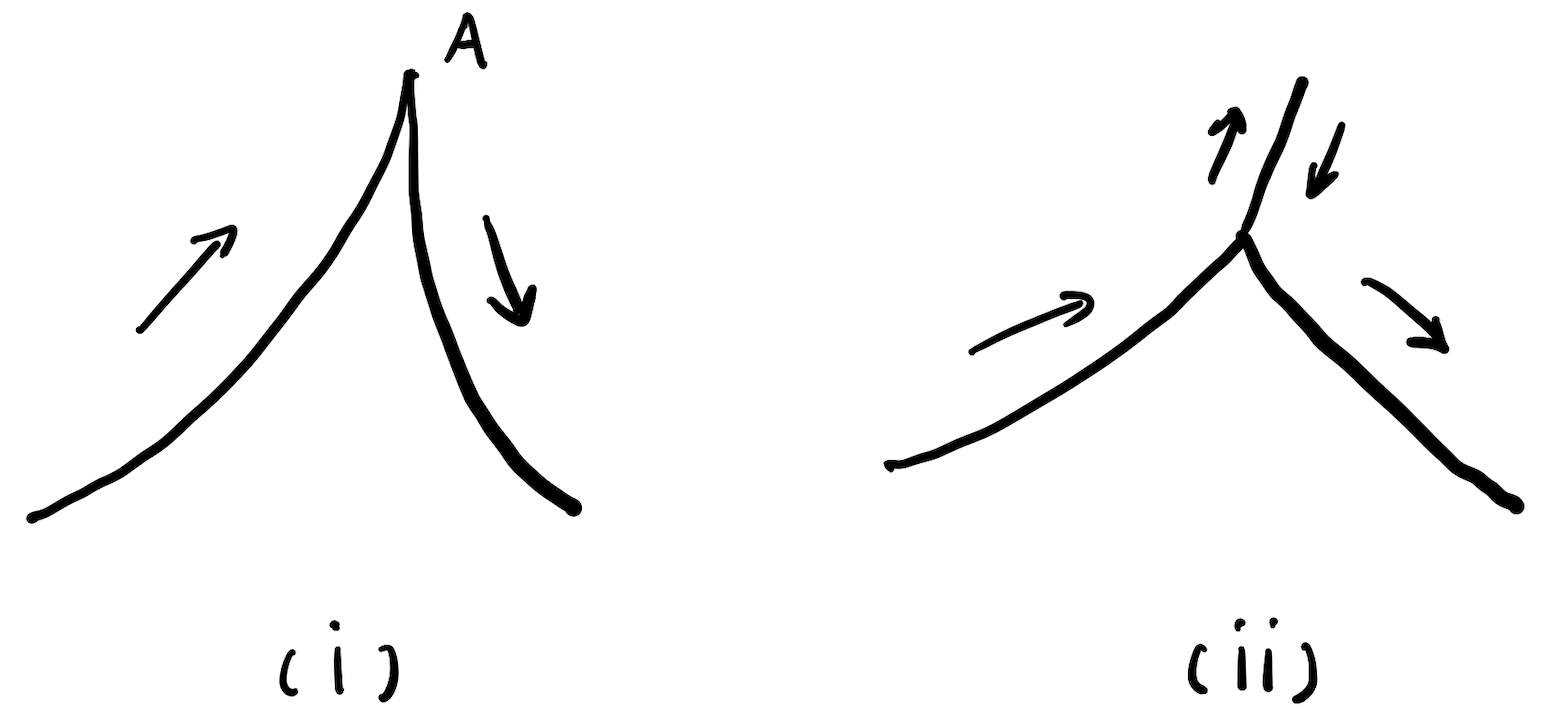} 
\protect\caption{Tree-reduced and non-tree-reduced paths.} \label{fig:RCusp} 
\end{center}
\end{figure}

From now on, we consider a path $\gamma:[0,L]\rightarrow\mathbb{R}^{2}$
given by (\ref{eq:DefPath}) where $\beta:[0,L]\rightarrow\mathbb{R}$
is a given measurable function. It is clear that $\gamma$ is parametrised
at unit speed, and $L$ is the length of $\gamma$.

A natural idea to capture the tree-reduced property in terms of the
angular path $\beta$ is to require that $\beta_{t}$ locally takes
values in an interval of length less than $\pi.$ This rules out the
possibility that $\gamma$ turns around to cancel itself. On the other
hand, making a $\pi$-turn does not necessarily produce a tree-like
piece, as illustrated by the cusp path in Figure \ref{fig:RCusp}
(i). In order to include this possibility, a natural extension is
to make the above requirement on $\beta$ hold outside an arbitrarily
small interval that contains the cusp singularity.

To make the above idea precise, we first introduce the following definition.
\begin{defn}
\label{def:Cusp}Let $\gamma:[0,L]\rightarrow\mathbb{R}^{2}$ be a
path given by (\ref{eq:DefPath}). We say that $\gamma$ is a \textit{regular}
\textit{cusp}, if the following two conditions hold:

\vspace{2mm} \noindent (i) there is a real number $a\in\mathbb{R}$
such that 
\[
\beta_{t}\in[a,a+\pi]\ \ \ {\rm for\ a.a.}\ t\in[0,L];
\]
(ii) for any $\delta>0,$ there is a closed subset $F\subseteq L$
which is a finite disjoint union of closed intervals, as well as two
real numbers $a_{\delta}>a$, $b_{\delta}<a+\pi$, such that
\[
\mu(F^{c})<\delta\ \text{and }\beta_{t}\in[a_{\delta},b_{\delta}]\ \ \ {\rm for\ a.a\ }t\in F.
\]
Here $\mu$ denotes the Lebesgue measure.
\end{defn}
\noindent  Since the definition is only concerned with the angular
path $\beta,$ we sometimes simply say that $\beta$ is a regular
cusp. The typical shape of a regular cusp is illustrated by Figure
\ref{fig:RCusp} (i). We mention that there is another type of cusps
that is more singular in terms of detecting the tree-reduced property
and is thus harder to deal with. We discuss this case in Section \ref{sec:MCusp}
(cf. Theorem \ref{thm:MCusp} below).
\begin{example}
\label{exa:NoCusp}A special situation of Definition \ref{def:Cusp}
is when 
\begin{equation}
\beta_{t}\in[a,b]\ \ \ {\rm for\ a.a.}\ t\in[0,L]\label{eq:GloCond}
\end{equation}
for some $a,b\in\mathbb{R}$ satisfying $b-a<\pi$. In this case,
there is no cusp singularities and the conditions in Definition \ref{def:Cusp}
are satisfied trivially.
\end{example}
Note that Definition \ref{def:Cusp} is global, while producing a
tree-like piece or not is a local issue. To capture the tree-reduced
property, it is more natural to localise Definition \ref{def:Cusp}.
This leads to the following definition, which will be assumed throughout
the rest of the present article.
\begin{defn}
\label{def:SRed}We say that $\gamma$ is \textit{strongly tree-reduced},
if for any $t\in(0,L),$ there exists a neighbourhood $(u_{t},v_{t})$
of $t$ such that $\gamma|_{[u_{t},v_{t}]\cap[0,L]}$ is a regular
cusp.
\end{defn}
Our main result is stated as follows.
\begin{thm}
\label{thm:Mthm}Let $\gamma:[0,L]\rightarrow\mathbb{R}^{2}$ be a
path given by (\ref{eq:DefPath}). Suppose that $\gamma$ is strongly
tree-reduced. Then the signature asymptotics formula (\ref{eq:LengConj})
holds.
\end{thm}
Theorem \ref{thm:Mthm} contains the case of ${\cal C}^{1}$-paths
at unit speed as a particular example. Indeed, if the angular path
$\beta_{t}$ is continuous, for each given $t$ the path $\beta_{s}$
takes values in an interval of length less than $\pi$ when $s$ is
near $t$. The result also contains the case of piecewise ${\cal C}^{1}$-paths
whose intersection angles are strictly less than $\pi$. For the situation
where the intersection angle is $\pi$, Theorem \ref{thm:Mthm} still
applies if the cusp singularity has the nature of Definition \ref{def:Cusp}.
Also see Section \ref{sec:MCusp} below for adapting the analysis
to the case of even more singular cusps.

On the other hand, the property stated in Definition \ref{def:SRed}
is a condition that captures the tree-reduced property and has no
implication on the regularity of the path $\gamma$. Our analysis
relies purely on such a non-degeneracy condition. In contrast to the
analysis developed in \cite{HL10,LX15}, regularity assumptions on
$\beta$ are not relevant here and measurability is sufficient for
our purpose.
\begin{rem}
It is also interesting to point out that, although heuristically convincing,
it is not at all obvious to directly show that strong tree-reducedness
implies tree-reducedness in the sense of Definition \ref{def:TRed}.
This is however an immediate consequence of Theorem \ref{thm:Mthm}
and the uniqueness theorem in \cite{BGLY16}.
\end{rem}
To summarise the main idea in our strategy, we consider the development
of $\gamma$ into the special linear group ${\rm SL}_{2}(\mathbb{R})$
which acts on the plane $\mathbb{R}^{2}$ in the canonical way. The
core of our approach is to look at the action of $\Gamma_{L}$ (where
$(\Gamma_{t})_{0\leqslant t\leqslant L}$ is the development of $\gamma$)
from a dynamical viewpoint, and to carefully examine the behaviour
of the associated angle dynamics at a microscopic level. We will elaborate
this point more precisely as we develop the analysis in the following
sections.

\section{\label{sec:PathDev}${\rm SL}_{2}(\mathbb{R})$-developments and
the associated ODE dynamics}

Our starting point of proving Theorem \ref{thm:Mthm} is to develop
the path $\gamma$ onto a suitably chosen Lie group from Cartan's
perspective. In this section, we first recall the general construction
of path developments under the framework of \cite{BGS20}. We then
specialise the development to a particular Lie group and establish
the associated ODE dynamics. The analysis of this ODE dynamics is
the core ingredient in our approach, which will be performed in later
sections.

\subsection{\label{subsec:CartanDev}The Cartan development of rough paths and
the intermediate lower estimate}

Let $V$ be a finite dimensional normed vector space. Let $G$ be
a finite dimensional Lie group with Lie algebra $\mathfrak{g}$. Suppose
that $F:V\rightarrow\mathfrak{g}$ is a given linear map, and $\rho:\mathfrak{g}\rightarrow{\rm End}(W)$
is a given representation (i.e. a Lie algebra homomorphism) of $\mathfrak{g}$
over a finite dimensional normed vector space $W$. Set $\Phi\triangleq\rho\circ F:V\rightarrow{\rm End}(W)$.
Let ${\bf X}=({\bf X}_{t})_{0\leqslant t\leqslant L}$ be a geometric
rough path over $V$ (cf. \cite{LQ02}).
\begin{defn}
The \textit{Cartan development} of ${\bf X}$ \textit{onto }$G$ under
$F$ is the solution to the differential equation
\[
\begin{cases}
dG_{t}=G_{t}\cdot F(d{\bf X}_{t}), & 0\leqslant t\leqslant L,\\
G_{0}=e,
\end{cases}
\]
where $e$ is the identity of $G$. Respectively, the \textit{Cartan
development of ${\bf X}$ onto }${\rm Aut}(W)$ under $\Phi$ is the
solution to the linear differential equation
\[
\begin{cases}
d\Gamma_{t}=\Gamma_{t}\cdot\Phi(d{\bf X}_{t}), & 0\leqslant t\leqslant L,\\
\Gamma_{0}={\rm Id}.
\end{cases}
\]
\end{defn}
\begin{rem}
There are two important reasons for considering path developments.
The first one is that, the end point $\Gamma_{L}$ of the development,
when one varies $\mathfrak{g}$ and the representation $\rho$ in
a suitably chosen class, should encode essentially all information
about the original rough path $\mathbf{X}$ (up to tree-like pieces).
The second one is that, the development is defined by a ``linear''
equation, hence linearising the analysis of nonlinear functionals
on path space. These two points are similar to the philosophy of working
with the signature of $\mathbf{X}$, but it allows much richer algebraic
structures through choosing the Lie algebras and representations properly.
This philosophy has not yet been fully explored in the literature,
and could be of potential interest for further applications in the
study of rough paths and stochastic processes.
\end{rem}
We now recall a general lower estimate proved in \cite{BGS20}. We
only state the version for bounded variation paths. Let $\gamma:[0,L]\rightarrow V$
be a continuous path with finite length. Recall that $L_{1}(\gamma)$
is the functional defined by the normalised signature asymptotics
(\ref{eq:LimFunc}). Under the above set-up, for each $\lambda>0,$
let $(\Gamma_{t}^{\lambda})_{0\leqslant t\leqslant L}$ be the Cartan
development of $\lambda\cdot\gamma$ onto ${\rm Aut}(W)$ under $\Phi$.
\begin{prop}
\label{prop:IntLower}The quantity $L_{1}(\gamma)$ admits the following
lower estimate:
\[
L_{1}(\gamma)\geqslant\underset{\lambda\rightarrow\infty}{\overline{\lim}}\frac{\log\|\Gamma_{L}^{\lambda}\|_{W\rightarrow W}}{\lambda\|\Phi\|_{V\rightarrow{\rm End}(W)}},
\]
where $\|\cdot\|_{X\rightarrow Y}$ denotes the operator norm between
Banach spaces $X$ and $Y$.
\end{prop}

\subsection{\label{subsec:SLDev}${\rm SL}_{2}(\mathbb{R})$-developments}

We now specialise our study to the context of Theorem \ref{thm:Mthm}.
In particular, let $V=\mathbb{R}^{2}$ whose canonical basis is denoted
as $\{{\rm e}_{1},{\rm e}_{2}\}$. Let $\gamma:[0,L]\rightarrow V$
be a continuous path with finite length $L$, parametrised at unit
speed. More explicitly, the path $\gamma$ is defined by the integral
(\ref{eq:DefPath}), where the angular path $\beta:[0,L]\rightarrow\mathbb{R}^{2}$
is a given measurable function.

Under the setting of Section \ref{subsec:CartanDev}, we choose $G={\rm SL}_{2}(\mathbb{R}),$
the space of real $2\times2$ matrices with determinant one. The group
${\rm SL}_{2}(\mathbb{R})$ is a three dimensional Lie group, whose
Lie algebra is given by $\mathfrak{g}=\mathfrak{sl}_{2}(\mathbb{R})$,
the space of real $2\times2$ matrices with zero trace. The Lie algebra
$\mathfrak{sl}_{2}(\mathbb{R})$ admits a Lorentzian metric defined
by 
\[
\langle A,B\rangle\triangleq\frac{1}{2}{\rm Tr}(AB),\ \ \ A,B\in\mathfrak{sl}_{2}(\mathbb{R}),
\]
which has signature $(+,+,-)$. An element $A\in\mathfrak{sl}_{2}(\mathbb{R})$
is \textit{hyperbolic}/\textit{elliptic}/\textit{parabolic} if $\langle A,A\rangle$
is positive/negative/zero. An orthonormal basis of $\mathfrak{sl}_{2}(\mathbb{R})$
under the Lorentzian metric $\langle\cdot,\cdot\rangle$ is given
by 
\[
E_{1}=\left(\begin{array}{cc}
1 & 0\\
0 & -1
\end{array}\right),\ E_{2}=\left(\begin{array}{cc}
0 & 1\\
1 & 0
\end{array}\right),\ E_{3}=\left(\begin{array}{cc}
0 & 1\\
-1 & 0
\end{array}\right).
\]
Note that $E_{1},E_{2}$ are hyperbolic elements and $E_{3}$ is elliptic.

We define the linear map $F:V\rightarrow\mathfrak{sl}_{2}(\mathbb{R})$
explicitly by mapping the basis $\{{\rm e}_{1},{\rm e}_{2}\}$ to
the hyperbolic elements $\{E_{1},E_{2}\}$ respectively, i.e. $F({\rm e}_{i})\triangleq E_{i}$
($i=1,2$).

Finally, the representation $\rho:\mathfrak{sl}_{2}(\mathbb{R})\rightarrow{\rm End}(W)$
is taken to be the canonical matrix representation, i.e. the action
of $\mathfrak{g}$ on $W=\mathbb{R}^{2}$ via matrix multiplication.
We assume that $W$ is also equipped with the Euclidean norm. Note
that the group ${\rm SL}_{2}(\mathbb{R})$ also acts on $W$ via matrix
multiplication. If we view $\mathrm{SL}_{2}(\mathbb{R})$ as a subspace
of the matrix algebra ${\rm Mat}_{2}(\mathbb{R})\cong{\rm End}(W),$
Cartan developments onto ${\rm SL}_{2}(\mathbb{R})$ and ${\rm Aut}(W)$
are identical. For each $\lambda>0,$ the Cartan development of $\lambda\cdot\gamma$
onto ${\rm SL}_{2}(\mathbb{R})$ is denoted as $\Gamma^{\lambda}=(\Gamma_{t}^{\lambda})_{0\leqslant t\leqslant L}$.
More explicitly, $\Gamma^{\lambda}$ satisfies the differential equation
\[
\frac{d\Gamma_{t}^{\lambda}}{dt}=\lambda\Gamma_{t}^{\lambda}\cdot\left(\begin{array}{cc}
\cos\beta_{t} & \sin\beta_{t}\\
\sin\beta_{t} & -\cos\beta_{t}
\end{array}\right),\ \ \ \Gamma_{0}^{\lambda}={\rm Id}.
\]
Note that such equation needs to be understood in the integral sense
or in the $t$-a.e. sense.

Under the above specific choice, Proposition \ref{prop:IntLower}
yields the following lower estimate for the quantity $L_{1}(\gamma)$:
\begin{equation}
L_{1}(\gamma)\geqslant\underset{\lambda\rightarrow\infty}{\overline{\lim}}\frac{\log\|\Gamma_{L}^{\lambda}\|_{\mathbb{R}^{2}\rightarrow\mathbb{R}^{2}}}{\lambda}.\label{eq:IntLow}
\end{equation}
To see this, we only need to check that the operator $\Phi=\rho\circ F:\mathbb{R}^{2}\rightarrow{\rm End}(\mathbb{R}^{2})$
has norm one. But this follows from the relation
\[
\big|\Phi(v)(w)\big|=|v|\cdot|w|\ \ \ \forall v,w\in\mathbb{R}^{2},
\]
which can be verified explicitly.
\begin{rem}
From the geometric viewpoint, it is also natural to consider the action
of ${\rm SL}_{2}(\mathbb{R})$ on the upper half plane $\mathbb{H}$
via M\"obius transformation, since ${\rm SL}_{2}(\mathbb{R})$ is
the isometry group of $\mathbb{H}$ when $\mathbb{H}$ is equipped
with the Lobachevsky hyperbolic metric. In this case, the action of
$\Gamma_{t}^{\lambda}$ on $\mathbb{H}$ gives the hyperbolic development
of $\gamma$ (cf. \cite{HL10} for an equivalent hyperbolic framework).
However, we do not take this geometric viewpoint and work with linear
actions instead.
\end{rem}

\subsection{The decoupled ODE system and the associated angle dynamics}

In view of (\ref{eq:IntLow}), in order to obtain a sharp lower bound
for $L_{1}(\gamma)$, we need to estimate $\|\Gamma_{L}^{\lambda}\|_{\mathbb{R}^{2}\rightarrow\mathbb{R}^{2}}$
effectively. For this purpose, we look at the action of $\Gamma_{L}^{\lambda}$
from a dynamical perspective which we now describe.

We introduce the notation $\Gamma_{s,t}^{\lambda}$ ($t\in[s,L]$)
to denote the Cartan development of $\lambda\cdot\gamma|_{[s,L]}$
evaluated at time $t.$ Simple calculation shows that 
\begin{equation}
\Gamma_{s,u}^{\lambda}=\Gamma_{s,t}^{\lambda}\cdot\Gamma_{t,u}^{\lambda}\ \ \ \forall s\leqslant t\leqslant u.\label{eq:SemG}
\end{equation}
As a result, for a given initial vector $\xi^{\lambda}\in\mathbb{R}^{2},$
the action $\Gamma_{L}^{\lambda}\xi^{\lambda}$ can be studied through
the following dynamical perspective:
\begin{equation}
\Gamma_{L}^{\lambda}\xi^{\lambda}=\Gamma_{t_{0},t_{1}}^{\lambda}\cdot\Gamma_{t_{1},t_{2}}^{\lambda}\cdot\cdots\cdot\Gamma_{t_{n-1},t_{n}}^{\lambda}\xi^{\lambda},\label{eq:DisDyn}
\end{equation}
where ${\cal P}=\{t_{i}\}_{0\leqslant i\leqslant n}$ is an arbitrarily
fine partition of $[0,L]$. The dynamics (\ref{eq:DisDyn}) reduces
to the following simple equation when we take ${\rm mesh}({\cal P})\rightarrow0.$
\begin{lem}
\label{lem:ODEw}Let $w_{t}^{\lambda}\triangleq\Gamma_{L-t,L}^{\lambda}\xi^{\lambda}.$
Then $w_{L}^{\lambda}=\Gamma_{L}^{\lambda}\xi^{\lambda},$ and $(w_{t}^{\lambda})_{0\leqslant t\leqslant L}$
is the unique solution to the differential equation:
\begin{equation}
\begin{cases}
\frac{dw_{t}^{\lambda}}{dt}=\lambda\left(\begin{array}{cc}
\cos\beta_{L-t} & \sin\beta_{L-t}\\
\sin\beta_{L-t} & -\cos\beta_{L-t}
\end{array}\right)\cdot w_{t}^{\lambda}, & 0\leqslant t\leqslant L,\\
w_{0}^{\lambda}=\xi^{\lambda}.
\end{cases}\label{eq:ODEw}
\end{equation}
\end{lem}
\begin{proof}
Let $t$ be given and $h>0.$ According to the relation (\ref{eq:SemG}),
we have 
\[
w_{t+h}^{\lambda}=\Gamma_{L-t-h,L-t}^{\lambda}\cdot w_{t}^{\lambda}.
\]
It follows from the equation of the Cartan development that 
\begin{align*}
\frac{w_{t+h}^{\lambda}-w_{t}^{\lambda}}{h} & =\frac{\Gamma_{L-t-h,L-t}^{\lambda}-{\rm Id}}{h}\cdot w_{t}^{\lambda}\\
 & =\frac{\lambda}{h}\int_{L-t-h}^{L-t}\Gamma_{L-t-h,u}^{\lambda}\left(\begin{array}{cc}
\cos\beta_{u} & \sin\beta_{u}\\
\sin\beta_{u} & -\cos\beta_{u}
\end{array}\right)\cdot w_{t}^{\lambda}du,
\end{align*}
The result follows by letting $h\rightarrow0^{+}.$
\end{proof}
\vspace{2mm} \noindent \textbf{Notation}. From now on, we denote
$\alpha_{t}\triangleq\beta_{L-t}$. It is obvious that $\alpha$ satisfies
Definition \ref{def:SRed} if and only if $\beta$ does.

\vspace{2mm}Our next step is to rewrite the equation (\ref{eq:ODEw})
using polar coordinates. Let 
\[
w_{t}^{\lambda}=\rho_{t}^{\lambda}e^{i\phi_{t}^{\lambda}},\ \ \ t\in[0,L].
\]
We also write the initial vector as $\xi^{\lambda}=\rho_{0}^{\lambda}e^{i\phi_{0}^{\lambda}},$
where $\rho_{0}^{\lambda}=1$ and the angle $\phi_{0}^{\lambda}$
is given fixed.
\begin{lem}
The pair $(\rho_{t}^{\lambda},\phi_{t}^{\lambda})$ satisfies the
following ODE system:\begin{numcases}{}
\frac{d\rho_{t}^{\lambda}}{dt}=\lambda\rho_{t}^{\lambda}\cos(\alpha_{t}-2\phi_{t}^{\lambda}),\label{eq:REqn}\\ 
\frac{d\phi_{t}^{\lambda}}{dt}=\lambda\sin(\alpha_{t}-2\phi_{t}^{\lambda}).\label{eq:AEqn} 
\end{numcases}
\end{lem}
\begin{proof}
Firstly, note that we have
\begin{align}
dw_{t}^{\lambda} & =d\rho_{t}^{\lambda}\cdot\left(\begin{array}{c}
\cos\phi_{t}^{\lambda}\\
\sin\phi_{t}^{\lambda}
\end{array}\right)+\rho_{t}^{\lambda}d\phi_{t}^{\lambda}\cdot\left(\begin{array}{c}
-\sin\phi_{t}^{\lambda}\\
\cos\phi_{t}^{\lambda}
\end{array}\right)\nonumber \\
 & =\left(\begin{array}{cc}
\cos\phi_{t}^{\lambda} & -\sin\phi_{t}^{\lambda}\\
\sin\phi_{t}^{\lambda} & \cos\phi_{t}^{\lambda}
\end{array}\right)\cdot\left(\begin{array}{c}
d\rho_{t}^{\lambda}\\
\rho_{t}^{\lambda}d\phi_{t}^{\lambda}
\end{array}\right).\label{eq:dwPolar}
\end{align}
In addition, according to the equation (\ref{eq:ODEw}), we have 
\begin{equation}
dw_{t}^{\lambda}=\lambda\rho_{t}^{\lambda}dt\cdot\left(\begin{array}{cc}
\cos\alpha_{t} & \sin\alpha_{t}\\
\sin\alpha_{t} & -\cos\alpha_{t}
\end{array}\right)\cdot\left(\begin{array}{c}
\cos\phi_{t}^{\lambda}\\
\sin\phi_{t}^{\lambda}
\end{array}\right).\label{eq:dwDev}
\end{equation}
By comparing (\ref{eq:dwPolar}) and (\ref{eq:dwDev}), we arrive
at 
\[
\left(\begin{array}{c}
d\rho_{t}^{\lambda}\\
\rho_{t}^{\lambda}d\phi_{t}^{\lambda}
\end{array}\right)=\lambda\rho_{t}^{\lambda}dt\cdot\left(\begin{array}{c}
\cos(\alpha_{t}-2\phi_{t}^{\lambda})\\
\sin(\alpha_{t}-2\phi_{t}^{\lambda})
\end{array}\right),
\]
which yields the desired ODE system.
\end{proof}
It is clear that the angular path $\phi_{t}^{\lambda}$ is absolutely
continuous. Observe that the ODE system for $(\rho_{t}^{\lambda},\phi_{t}^{\lambda})$
is decoupled, in the sense that the angular equation \eqref{eq:AEqn}
does not depend on $\rho_{t}^{\lambda}$. In addition, by linearity
the radial component $\rho_{t}^{\lambda}$ can be easily solved from
the radial equation \eqref{eq:REqn} (recall $\rho_{0}^{\lambda}=1$)
as
\begin{equation}
\rho_{t}^{\lambda}=\exp\big(\lambda\int_{0}^{t}\cos(\alpha_{s}-2\phi_{s}^{\lambda})ds\big),\ \ \ t\in[0,L].\label{eq:Rad}
\end{equation}

\begin{rem}
The following viewpoint can be taken to avoid ambiguity when choosing
the angular component $\phi_{t}^{\lambda}$. Given the initial angle
$\phi_{0}^{\lambda}$, the angular equation \eqref{eq:AEqn} admits
a unique solution $\phi_{t}^{\lambda}.$ Define $\rho_{t}^{\lambda}$
by the formula (\ref{eq:Rad}) accordingly. Then $w_{t}^{\lambda}\triangleq\rho_{t}^{\lambda}e^{i\phi_{t}^{\lambda}}$
gives the solution to the equation (\ref{eq:ODEw}) which clearly
coincides with $(\Gamma_{L-t,L}^{\lambda}\xi^{\lambda})_{0\leqslant t\leqslant L}$
by uniqueness.
\end{rem}
Since $\rho_{L}^{\lambda}=|\Gamma_{L}^{\lambda}\xi^{\lambda}|$ and
$\xi^{\lambda}$ is assumed to be a unit vector, we have the following
important lemma which is a direct consequence of the estimate (\ref{eq:IntLow}).
\begin{lem}
\label{lem:IntLow}The quantity $L_{1}(\gamma)$ satisfies 
\begin{equation}
L_{1}(\gamma)\geqslant\underset{\lambda\rightarrow\infty}{\overline{\lim}}\int_{0}^{L}\cos(\alpha_{t}-2\phi_{t}^{\lambda})dt,\label{eq:IntegralLow}
\end{equation}
where the initial angle $\phi_{0}^{\lambda}\in\mathbb{R}$ is arbitrarily
given.
\end{lem}
\noindent Lemma \ref{lem:IntLow} provides a clearer picture of proving
Theorem \ref{thm:Mthm}. In order to produce the lower bound of $L_{1}(\gamma)$
given by the length $L$, we are led to showing that \textit{the angular
path $2\phi_{t}^{\lambda}$ is close to $\alpha_{t}$ for most of
the time when $\lambda$ is large}. At a heuristic level, this phenomenon
is reasonable, since the angular equation \eqref{eq:AEqn} is suggesting
a strong mean-reversing behaviour of $2\phi_{t}^{\lambda}$ towards
the path $\alpha_{t}$ when $\lambda$ is large. However, making this
phenomenon mathematically precise is a non-trivial challenging task,
since no regularity assumptions are made on the path $\alpha_{t}$
and one has to rely on fine measure-theoretic arguments. The analysis
simplifies substantially if $\alpha_{t}$ is assumed to be a continuous
function (cf. Appendix for a discussion on this case).

\section{\label{sec:GloLem}Some important lemmas on the global behaviour
of the angle dynamics}

The core of our approach is to analyse the angle dynamics for $\phi_{t}^{\lambda}$.
In this section, we derive several key lemmas on the behaviour of
$\phi_{t}^{\lambda}$ relative to the path $\alpha_{t}$ under global
assumptions on $\alpha_{t}$. These results will be used in a localised
situation when we prove the main theorem in the next section.

The first lemma tells us that $2\phi_{t}^{\lambda}$ remains in the
same range as $\alpha_{t}$'s provided that the initial angle $2\phi_{0}^{\lambda}$
does.
\begin{lem}
\label{lem:RangePhi}Let $a\in\mathbb{R}.$ Suppose that $\alpha_{t}\in[a,a+\pi]$
for a.a. $t\in[0,L].$ If $2\phi_{0}^{\lambda}\in(a,a+\pi),$ then
$2\phi_{t}^{\lambda}\in(a,a+\pi)$ for every $t\in[0,L]$ and $\lambda>0.$
\end{lem}
\begin{proof}
Recall that $2\phi_{t}^{\lambda}$ is absolutely continuous. Set $b\triangleq a+\pi.$
Suppose on the contrary that $2\phi_{t}^{\lambda}$ leaves $(a,a+\pi)$
at some time, and let us assume that $2\phi_{t}^{\lambda}$ hits the
end point $b$ before hitting $a$. Define 
\[
t_{2}\triangleq\inf\big\{ t:2\phi_{t}^{\lambda}=b\big\},\ t_{1}\triangleq\sup\big\{ t<t_{2}:2\phi_{t}^{\lambda}<\frac{a+b}{2}\big\}.
\]
We take $t_{1}=0$ if $2\phi_{t}^{\lambda}\geqslant\frac{a+b}{2}$
for all $t<t_{2}$. Apparently, we have $2\phi_{t_{1}}^{\lambda}<b,$
$2\phi_{t_{2}}^{\lambda}=b$ and $2\phi_{t}^{\lambda}\in[\frac{a+b}{2},b]$
for all $t\in[t_{1},t_{2}]$. By using the equation \eqref{eq:AEqn}
of $\phi_{t}^{\lambda},$ for a.a. $t\in[t_{1},t_{2}]$ we have
\begin{align*}
\frac{d\phi_{t}^{\lambda}}{dt} & =\lambda\sin(\alpha_{t}-2\phi_{t}^{\lambda})\\
 & =\lambda\sin(\alpha_{t}-2\phi_{t}^{\lambda}){\bf 1}_{\{a\leqslant\alpha_{t}<\frac{a+b}{2}\}}+\lambda\sin(\alpha_{t}-2\phi_{t}^{\lambda}){\bf 1}_{\{\frac{a+b}{2}\leqslant\alpha_{t}\leqslant b\}}.
\end{align*}
In the first region, we know that $\alpha_{t}-2\phi_{t}^{\lambda}\in[-\pi,0]$
and thus the sine function is non-positive. In the second region,
note that both of $\alpha_{t}-2\phi_{t}^{\lambda}$ and $b-2\phi_{t}^{\lambda}$
lie in $[-\frac{\pi}{2},\frac{\pi}{2}]$. As a result, we have
\[
\sin(\alpha_{t}-2\phi_{t}^{\lambda})\leqslant\sin(b-2\phi_{t}^{\lambda})
\]
in this region. Note moreover that $b-2\phi_{t}^{\lambda}\in[0,\frac{\pi}{2}]$
for $t\in[t_{1},t_{2}]$. It follows that 
\[
\frac{d\phi_{t}^{\lambda}}{dt}\leqslant\lambda\sin(b-2\phi_{t}^{\lambda}){\bf 1}_{\{\frac{a+b}{2}\leqslant\alpha_{t}\leqslant b\}}\leqslant\lambda(b-2\phi_{t}^{\lambda}).
\]
Equivalently, we have 
\[
\frac{d}{dt}\big(2e^{2\lambda t}\phi_{t}^{\lambda}\big)\leqslant2\lambda be^{2\lambda t}.
\]
By integrating the above inequality over $[t_{1},t_{2}],$ we arrive
at 
\[
2\phi_{t_{2}}^{\lambda}\leqslant b-(b-2\phi_{t_{1}}^{\lambda})e^{-2\lambda(t_{2}-t_{1})}<b,
\]
which is a contradiction. A similar argument also leads to a contradiction
in the case when $2\phi_{t}^{\lambda}$ hits $a$ before $b$. Therefore,
we conclude that $2\phi_{t}^{\lambda}\in(a,b)$ for all time.
\end{proof}
The next lemma quantifies how much $2\phi_{t}^{\lambda}$ can deviate
from an interval $[a,b]$ if the path $\alpha_{t}$ does not always
stay in this interval.
\begin{lem}
\label{lem:DeviPhi}Let $a,b\in\mathbb{R}$ be such that $0<b-a<\pi.$
Define 
\[
r\triangleq2\lambda\mu\big(\{t:\alpha_{t}\notin[a,b]\}\big),
\]
Suppose that $b-a+r<\pi$ and $2\phi_{0}^{\lambda}\in[a,b].$ Then
\[
2\phi_{t}^{\lambda}\in[a-r,b+r]\ \ \ \forall t\in[0,L].
\]
\end{lem}
\begin{proof}
Suppose on the contrary that $2\phi_{t}^{\lambda}$ exits the interval
$[a-r,b+r]$ at some time, and assume that it exits at the end point
$b+r.$ We can then find a time $\tau$ such that 
\[
b+r<2\phi_{\tau}^{\lambda}<a+\pi.
\]
Define 
\[
t_{2}\triangleq\inf\{t<\tau:2\phi_{t}^{\lambda}=2\phi_{\tau}^{\lambda}\},\ t_{1}\triangleq\sup\{t<t_{2}:2\phi_{t}^{\lambda}\in[a,b]\},
\]
and write 
\[
A\triangleq\{t:\alpha_{t}\in[a,b]\}.
\]
Note that $2\phi_{t_{1}}^{\lambda}=b.$ It follows from the angular
equation \eqref{eq:AEqn} that 
\begin{align*}
2\phi_{t_{2}}^{\lambda}-b & =-2\lambda\int_{t_{1}}^{t_{2}}\sin(2\phi_{t}^{\lambda}-\alpha_{t})dt\\
 & =-2\lambda\int_{[t_{1},t_{2}]\cap A}\sin(2\phi_{t}^{\lambda}-\alpha_{t})dt-2\lambda\int_{[t_{1},t_{2}]\cap A^{c}}\sin(2\phi_{t}^{\lambda}-\alpha_{t})dt\\
 & \leqslant-2\lambda\int_{[t_{1},t_{2}]\cap A^{c}}\sin(2\phi_{t}^{\lambda}-\alpha_{t})dt\\
 & \leqslant2\lambda\mu([t_{1},t_{2}]\cap A^{c}).
\end{align*}
Therefore, we have
\[
b+r<2\phi_{\tau}^{\lambda}=2\phi_{t_{2}}^{\lambda}\leqslant b+2\lambda\mu([t_{1},t_{2}]\cap A^{c})\leqslant b+r,
\]
which is a contradiction. The case when $2\phi_{t}^{\lambda}$ exits
$[a-r,b+r]$ through the end point $a-r$ can be treated in a similar
way. Consequently, we conclude that $2\phi_{t}^{\lambda}\in[a-r,b+r]$
for all time.
\end{proof}
\begin{rem}
Heuristically, Lemma \ref{lem:DeviPhi} tells us that the longer $\alpha_{t}$
stays in $[a,b]$ (equivalently the smaller $r$ is), the less will
$2\phi_{t}^{\lambda}$ deviate from $[a,b].$ In the special case
when $\alpha_{t}\in[a,b]$ a.a. $t\in[0,L]$ (i.e. when $r=0$), we
have $2\phi_{t}^{\lambda}\in[a,b]$ for all time. But this conclusion
slightly weaker than Lemma \ref{lem:RangePhi} since we have assumed
$b<a+\pi$ here.
\end{rem}
The final lemma quantifies how fast $2\phi_{t}^{\lambda}$ gets attracted
to the region where $\alpha_{t}$ stays for most of the time, if initially
$2\phi_{0}^{\lambda}$ is far away from this region.
\begin{lem}
\label{lem:CatTim}Let $a<b$ be such that $b-a<\pi$. Let $c<d$
and $\varepsilon>0$ be such that 
\[
[c-\varepsilon,d+\varepsilon]\subseteq(a,b)\ \text{and }\varepsilon<\pi-(b-a).
\]
Suppose that $\alpha_{t}\in[a,b]$ for a.a. $t\in[0,L]$, and 
\begin{equation}
2\phi_{0}^{\lambda}\in(a,b)\backslash[c-\varepsilon,d+\varepsilon].\label{eq:BadIni}
\end{equation}
Define $B\triangleq\{t:\alpha_{t}\in[c,d]\}$ and 
\[
\tau\triangleq\inf\{t:2\phi_{t}^{\lambda}\in[c-\varepsilon,d+\varepsilon]\}.
\]
Then 
\begin{equation}
\tau\leqslant\frac{b-a}{2\lambda\sin\varepsilon}+\frac{1+\sin\varepsilon}{\sin\varepsilon}\mu(B^{c}).\label{eq:CatTim}
\end{equation}
\end{lem}
\begin{proof}
First of all, we know from Lemma \ref{lem:RangePhi} that $2\phi_{t}^{\lambda}\in[a,b]$
for all $t$. In view of the assumption (\ref{eq:BadIni}), suppose
that $2\phi_{0}^{\lambda}\in(d+\varepsilon,b).$ Then for a.a. $t\in[0,\tau],$
we have 
\begin{align*}
\frac{d}{dt}(2\phi_{t}^{\lambda}) & =-2\lambda\sin(2\phi_{t}^{\lambda}-\alpha_{t})\\
 & =-2\lambda\sin(2\phi_{t}^{\lambda}-\alpha_{t}){\bf 1}_{B}-2\lambda\sin(2\phi_{t}^{\lambda}-\alpha_{t}){\bf 1}_{B^{c}}.
\end{align*}
Note that for $t\in B\cap[0,\tau]$ we have
\[
\varepsilon\leqslant2\phi_{t}^{\lambda}-\alpha_{t}\leqslant b-a<\pi.
\]
Since $\varepsilon<\pi-(b-a),$ it follows that 
\[
\sin(2\phi_{t}^{\lambda}-\alpha_{t})\geqslant\sin\varepsilon\ \ \ {\rm on}\ B\cap[0,\tau].
\]
Therefore, 
\[
\frac{d}{dt}(2\phi_{t}^{\lambda})\leqslant-2\lambda(\sin\varepsilon){\bf 1}_{B}+2\lambda{\bf 1}_{B^{c}}\ \ \ {\rm a.a.\ }t\in[0,\tau].
\]
By integrating the above inequality, we obtain 
\begin{align*}
2\phi_{\tau}^{\lambda} & \leqslant2\phi_{0}^{\lambda}-2\lambda\sin\varepsilon\cdot\mu(B\cap[0,\tau])+2\lambda\mu(B^{c}\cap[0,\tau])\\
 & =2\phi_{0}^{\lambda}-2\lambda\sin\varepsilon\cdot(\tau-\mu(B^{c}\cap[0,\tau]))+2\lambda\mu(B^{c}\cap[0,\tau])\\
 & =2\phi_{0}^{\lambda}-2\lambda\tau\sin\varepsilon+2\lambda(1+\sin\varepsilon)\mu(B^{c}\cap[0,\tau]).
\end{align*}
Therefore, 
\begin{align*}
2\lambda\tau\sin\varepsilon & \leqslant2\phi_{0}^{\lambda}-2\phi_{\tau}^{\lambda}+2\lambda(1+\sin\varepsilon)\mu(B^{c}\cap[0,\tau])\\
 & \leqslant b-a+2\lambda(1+\sin\varepsilon)\mu(B^{c}\cap[0,\tau]).
\end{align*}
Rearranging the terms gives the estimate (\ref{eq:CatTim}). A similar
argument gives the same conclusion for the case $2\phi_{0}^{\lambda}\in(a,c-\varepsilon)$.
Note that the situation when $\tau=L$ (i.e. when $2\phi_{t}^{\lambda}$
never enters $[c-\varepsilon,d+\varepsilon]$) is included in the
above argument.
\end{proof}
\begin{rem}
Heuristically, Lemma \ref{lem:CatTim} tells us that, if $\alpha_{t}$
stays in $[c,d]$ for most of the time (i.e. $\mu(B^{c})$ is small)
and if $\lambda$ is large, it takes a short period of time for $2\phi_{t}^{\lambda}$
to enter the interval $[c-\varepsilon,d+\varepsilon]$ (i.e. $\tau$
is small).
\end{rem}
\begin{rem}
The fact that $b-a<\pi$ is critical to make use of the monotonicity
property of the sine function in the proof of Lemma \ref{lem:CatTim}.
The precise use of this lemma in the proof of the main theorem requires
a minor technical modification (cf. Lemma \ref{lem:EntryTimeEst}
below).
\end{rem}

\section{\label{sec:PfMthm}Proof of the main theorem}

In this section, we develop the proof of Theorem \ref{thm:Mthm}.
To make our strategy more transparent, we first prove the theorem
under the global assumption of regular cusps (cf. Definition \ref{def:Cusp}).
This part contains the essential idea of the proof. After that, we
localise the result to the context of strongly tree-reduced paths
(cf. Definition \ref{def:SRed}).

\subsection{\label{subsec:Glo}Proof of Theorem \ref{thm:Mthm}: the global case}

Suppose that $\gamma:[0,L]\rightarrow\mathbb{R}^{2}$ is a path defined
by (\ref{eq:DefPath}), where the angular path $\beta:[0,L]\rightarrow\mathbb{R}$
is a given measurable function. In this subsection, we aim at proving
the following result.
\begin{thm}
\label{thm:LConGlo}Suppose that $\gamma$ is a regular cusp in the
sense of Definition \ref{def:Cusp}. Then the signature asymptotics
formula (\ref{eq:LengConj}) holds.
\end{thm}
Vaguely speaking, our strategy is to analyse the local behaviour of
the angle dynamics \eqref{eq:AEqn} on each sub-interval of suitable
partitions of $[0,L],$ and then to examine how these microscopic
effects accumulate on the global scale. The analysis for the former
point is based on suitable localisation of the results obtained in
Section \ref{sec:GloLem}.

We now develop the precise details of the proof of Theorem \ref{thm:LConGlo}.
To better convey the logic and reasoning, we divide the argument into
several major steps. Recall that $\alpha_{t}\triangleq\beta_{L-t}.$
From Definition \ref{def:Cusp} (i), we know that the angular path
$\alpha_{t}$ satisfies 
\[
\alpha_{t}\in[a,a+\pi]\ \ \ {\rm for\ a.a.\ }t\in[0,L],
\]
where $a\in\mathbb{R}$ is given fixed. If one does not want to bother
with cusps, the argument below appears to be simpler under the assumption
that $\alpha_{t}\in[a,b]$ for a.a. $t$ where $b-a<\pi$ (cf. Example
\ref{exa:NoCusp}).

\subsubsection{\label{subsec:S1}Step one: localising the path $\alpha_{t}$}

Let $\delta>0$ be fixed. According to Definition \ref{def:Cusp}
(ii), there is a closed subset $F_{1}\subseteq[0,L]$, which is a
finite disjoint union of closed intervals, as well as two real numbers
$a_{\delta}>a,b_{\delta}<b\triangleq a+\pi$, such that $\mu(F_{1}^{c})<\delta$
and 
\[
\alpha_{t}\in[a_{\delta},b_{\delta}]\ \ \ {\rm for\ a.a.}\ t\in F_{1}.
\]

To proceed further, we first recall the classical Lusin's theorem
(cf. Folland \cite{Folland99}) as follows.
\begin{thm}
Let $f:[p,q]\rightarrow\mathbb{C}$ be a Lebesgue measurable function.
Then for any $\eta>0$, there exists a compact set $E\subseteq[p,q],$
such that $\mu(E^{c})<\eta$ and $f|_{E}$ is continuous.
\end{thm}
\noindent  Let $\eta>0$ be another given number, and let $\varepsilon>0$
be such that 
\begin{equation}
\varepsilon<\min\big\{ a_{\delta}-a,b-b_{\delta},\frac{\pi}{4}\big\}.\label{eq:ChoEps}
\end{equation}
Note that $\varepsilon$ is independent of $\eta.$ According to the
above Lusin's theorem, we can choose a compact subset $F_{2}\subseteq F_{1},$
such that $\mu(F_{1}\backslash F_{2})<\eta$ and $\alpha|_{F_{2}}$
is (uniformly) continuous. As a result, there exists $\rho>0,$ such
that

\[
s,t\in F_{2},\ |t-s|<\rho\implies|\alpha_{t}-\alpha_{s}|<\varepsilon.
\]
Since $\alpha_{t}\in[a_{\delta},b_{\delta}]$ a.e. on $F_{1},$ by
further reducing $F_{2}$ if necessary, we may assume that 
\[
\alpha_{t}\in[a_{\delta},b_{\delta}]\ \ \ \text{for \textit{every} }t\in F_{2}.
\]

Recall that $F_{1}$ is a finite disjoint union of, say, $N_{\delta}$
closed intervals. Given $\lambda>0,$ we set 
\begin{equation}
n\triangleq[c\cdot\lambda]+1\ {\rm where}\ c\triangleq\varepsilon\sin\varepsilon.\label{eq:ChoPar}
\end{equation}
The reason for choosing this $c$ will be clear later on. We consider
the partition ${\cal P}_{n}=\{t_{i}^{n}\}_{0\leqslant i\leqslant n}$
of $F_{1}$ which divide each closed interval in $F_{1}$ into small
sub-intervals of equal length. When $\lambda$ (and thus $n$) is
large enough, we can ensure that 
\begin{equation}
{\rm mesh}{\cal P}_{n}=\frac{\mu(F_{1})}{n}<\rho.\label{eq:Mesh}
\end{equation}

\noindent \textbf{Note}. We have introduced several parameters $\delta,\varepsilon,\eta,\lambda$.
At some point later on, we will introduce one more independent parameter
$M.$ All these parameters need to pass to the limit in the last step.
It may be helpful to keep in mind that the following order of taking
limits will be implemented eventually:
\begin{equation}
\lambda\rightarrow\infty,\eta\rightarrow0^{+},M\rightarrow\infty,\varepsilon\rightarrow0^{+},\delta\rightarrow0^{+}.\label{eq:LimOrd}
\end{equation}

In what follows, we work with any given $\lambda>0$ that satisfies
(\ref{eq:Mesh}). This is legal in the spirit of (\ref{eq:LimOrd}),
since the first limiting procedure we will take is sending $\lambda\rightarrow\infty$.
For each $1\leqslant i\leqslant n,$ we write $I_{i}^{n}\triangleq[t_{i-1}^{n},t_{i}^{n}]$
and define 
\[
\alpha_{i}^{n}\triangleq\inf\{\alpha_{t}:t\in F_{2}\cap I_{i}^{n}\},\ \beta_{i}^{n}\triangleq\sup\{\alpha_{t}:t\in F_{2}\cap I_{i}^{n}\}.
\]
Note that 
\begin{equation}
0\leqslant\beta_{i}^{n}-\alpha_{i}^{n}<\varepsilon\ \text{and }\alpha_{i}^{n},\beta_{i}^{n}\in[a_{\delta},b_{\delta}].\label{eq:CtyEst}
\end{equation}

\subsubsection{Step two: the local behaviour of the angle dynamics}

Now we consider the ${\rm SL}_{2}(\mathbb{R})$-development of $\gamma_{t}$
constructed in Section \ref{subsec:SLDev}. Recall that the function
$\phi_{t}^{\lambda}$ satisfies the angular equation \eqref{eq:AEqn}.
We assume that $2\phi_{0}^{\lambda}\in(a,b).$ The core of our argument
concerns with understanding the local behaviour of $2\phi_{t}^{\lambda}$
on each sub-interval $I_{i}^{n}$ and its accumulated effect on the
global scale. In particular, there are two key points that we shall
establish in a precise way:

\vspace{2mm} \noindent (i) The time it takes $2\phi_{t}^{\lambda}$
($t\in I_{i}^{n}$) to enter the interval $[\alpha_{i}^{n}-\varepsilon,\beta_{i}^{n}+\varepsilon]$
adds up (over $i$) to a negligible quantity;\\
(ii) Once $2\phi_{t}^{\lambda}\in[\alpha_{i}^{n}-\varepsilon,\beta_{i}^{n}+\varepsilon]$
at some $t\in I_{i}^{n},$ the portion of $2\phi_{u}^{\lambda}$ on
$[t,t_{i}^{n}]$ provides a main contribution in the lower estimate
of the radial function $\rho_{t}^{\lambda}$ defined by (\ref{eq:Rad})
(or equivalently, the integral appearing in (\ref{eq:IntegralLow})).\\

We quantify these two points precisely in Step Three below. The main
ingredient in the current step is the following localised version
of Lemma \ref{lem:CatTim}. Let us introduce 
\[
\tau_{i}^{n}\triangleq\inf\{t\in I_{i}^{n}:2\phi_{t}^{\lambda}\in[\alpha_{i}^{n}-\varepsilon,\beta_{i}^{n}+\varepsilon]\},\ \ \ \sigma_{i}^{n}\triangleq\tau_{i}^{n}-t_{i-1}^{n}.
\]
The quantity $\sigma_{i}^{n}$ gives the amount of time within $I_{i}^{n}$
before $2\phi_{t}^{\lambda}$ enters the ``good'' region $[\alpha_{i}^{n}-\varepsilon,\beta_{i}^{n}+\varepsilon]$.
If $2\phi_{t_{i-1}^{n}}^{\lambda}\in[\alpha_{i}^{n}-\varepsilon,\beta_{i}^{n}+\varepsilon]$,
we trivially have $\sigma_{i}^{n}=0.$ Otherwise, we have the following
estimate for the time period $\sigma_{i}^{n}$, which is a minor adaptation
of Lemma \ref{lem:CatTim}.
\begin{lem}
\label{lem:EntryTimeEst}The quantity $\sigma_{i}^{n}$ satisfies
the following estimate:
\[
\sigma_{i}^{n}\leqslant\frac{\pi}{2\lambda\sin\varepsilon}+\frac{1+\sin\varepsilon}{\sin\varepsilon}\mu((F_{1}\backslash F_{2})\cap I_{i}^{n}).
\]
\end{lem}
\begin{proof}
We expect to apply Lemma \ref{lem:CatTim} to the context of 
\[
[a,b]=[a_{\delta},b_{\delta}],\ [c,d]=[\alpha_{i}^{n},\beta_{i}^{n}],\ [0,L]=I_{i}^{n}.
\]
However, the application is not entirely obvious, since we do not
know if $2\phi_{t_{i-1}^{n}}^{\lambda}\in[a_{\delta},b_{\delta}].$
The point is, we do know that $2\phi_{t}^{\lambda}\in[a,b]$ for all
$t$ (cf. Lemma \ref{lem:RangePhi}), and the previous proof of Lemma
\ref{lem:CatTim} remains valid under the requirement (\ref{eq:ChoEps}).
To elaborate this, we only consider the case when $2\phi_{t_{i-1}^{n}}^{\lambda}>\beta_{i}^{n}+\varepsilon$
as the other scenario is similar. We set
\[
B_{i}^{n}\triangleq\{t\in I_{i}^{n}:\alpha_{t}\in[\alpha_{i}^{n},\beta_{i}^{n}]\}.
\]
In the same way as in that proof, we have 
\[
\frac{d}{dt}(2\phi_{t}^{\lambda})\leqslant-2\lambda\sin(2\phi_{t}^{\lambda}-\alpha_{t}){\bf 1}_{B_{i}^{n}}+2\lambda{\bf 1}_{(B_{i}^{n})^{c}}
\]
for a.a. $t\in[t_{i-1}^{n},\tau_{i}^{n}].$ Note that $2\phi_{t}^{\lambda}\in[\beta_{i}^{n}+\varepsilon,b]$
on $[t_{i-1}^{n},\tau_{i}^{n}].$ Therefore, we have 
\[
\varepsilon\leqslant2\phi_{t}^{\lambda}-\alpha_{t}\leqslant b-\alpha_{i}^{n}\leqslant b-a_{\delta}\ \ \ {\rm on\ }B_{i}^{n}\cap[t_{i-1}^{n},\tau_{i}^{n}].
\]
By using the choice (\ref{eq:ChoEps}) of $\varepsilon$, we see that
\[
\sin(2\phi_{t}^{\lambda}-\alpha_{t})\geqslant\sin\varepsilon\ \ \ {\rm on\ }B_{i}^{n}\cap[t_{i-1}^{n},\tau_{i}^{n}].
\]
The rest of the argument is identical to the proof of Lemma \ref{lem:CatTim},
yielding the estimate 
\[
2\lambda\sigma_{i}^{n}\sin\varepsilon\leqslant2\phi_{t_{i-1}^{n}}^{\lambda}-2\phi_{\tau_{i}}^{\lambda}+2\lambda(1+\sin\varepsilon)\mu((B_{i}^{n})^{c}\cap I_{i}^{n}).
\]
Since $F_{2}\cap I_{i}^{n}\subseteq B_{i}^{n}$, we have 
\begin{align*}
2\lambda\sigma_{i}^{n}\sin\varepsilon & \leqslant2\phi_{t_{i-1}^{n}}^{\lambda}-2\phi_{\tau_{i}}^{\lambda}+2\lambda(1+\sin\varepsilon)\mu((F_{1}\backslash F_{2})\cap I_{i}^{n})\\
 & \leqslant b-a+2\lambda(1+\sin\varepsilon)\mu((F_{1}\backslash F_{2})\cap I_{i}^{n})\\
 & =\pi+2\lambda(1+\sin\varepsilon)\mu((F_{1}\backslash F_{2})\cap I_{i}^{n}).
\end{align*}
Rearranging the inequality gives the desired estimate.
\end{proof}

\subsubsection{Step three: the global estimate}

According to the intermediate lower estimate given by Lemma \ref{lem:IntLow},
our task is to estimate the integral 
\[
I_{\lambda}\triangleq\int_{0}^{L}\cos(2\phi_{t}^{\lambda}-\alpha_{t})dt
\]
from below when $\lambda$ is large. For this purpose, we first write
\begin{align*}
I_{\lambda} & =\int_{F_{2}^{c}}\cos(2\phi_{t}^{\lambda}-\alpha_{t})dt+\int_{F_{2}}\cos(2\phi_{t}^{\lambda}-\alpha_{t})dt\\
 & \geqslant-\mu(F_{1}^{c})-\mu(F_{1}\backslash F_{2})+\sum_{i=1}^{n}\int_{F_{2}\cap I_{i}^{n}}\cos(2\phi_{t}^{\lambda}-\alpha_{t})dt\\
 & \geqslant-\delta-\eta+\sum_{i=1}^{n}\int_{F_{2}\cap I_{i}^{n}}\cos(2\phi_{t}^{\lambda}-\alpha_{t})dt.
\end{align*}
To analyse the summation on the right hand side, we decompose it as
\[
\sum_{i=1}^{n}\int_{F_{2}\cap I_{i}^{n}}\cos(2\phi_{t}^{\lambda}-\alpha_{t})dt=J_{n}+K_{n},
\]
where 
\[
J_{n}\triangleq\sum_{i=1}^{n}\int_{F_{2}\cap[t_{i-1}^{n},\tau_{i}^{n}]}\cos(2\phi_{t}^{\lambda}-\alpha_{t})dt
\]
and 
\[
K_{n}\triangleq\sum_{i=1}^{n}\int_{F_{2}\cap[\tau_{i}^{n},t_{i}^{n}]}\cos(2\phi_{t}^{\lambda}-\alpha_{t})dt
\]
respectively.

For the term $J_{n},$ according to Lemma \ref{lem:EntryTimeEst}
and the choice (\ref{eq:ChoPar}) of $n$, we have 
\begin{align}
J_{n} & \geqslant-\sum_{i=1}^{n}\mu(F_{2}\cap[t_{i-1}^{n},\tau_{i}^{n}])\geqslant-\sum_{i=1}^{n}\sigma_{i}^{n}\nonumber \\
 & \geqslant-\frac{\pi n}{2\lambda\sin\varepsilon}-\frac{1+\sin\varepsilon}{\sin\varepsilon}\sum_{i=1}^{n}\mu((F_{1}\backslash F_{2})\cap I_{i}^{n})\nonumber \\
 & \geqslant-\frac{\pi\varepsilon}{2}-\frac{(1+\sin\varepsilon)\eta}{\sin\varepsilon}.\label{eq:WaitTime}
\end{align}

For the term $K_{n},$ we introduce an extra independent parameter
$M>0$. Define ${\cal B}_{n}$ to be the collection of those $i$'s
such that

\[
\mu((F_{1}\backslash F_{2})\cap I_{i}^{n})>\frac{M}{n}\eta,
\]
and set ${\cal G}_{n}\triangleq{\cal B}_{n}^{c}$. Then we have 
\[
\eta>\mu(F_{1}\backslash F_{2})=\sum_{i=1}^{n}\mu(((F_{1}\backslash F_{2}))\cap I_{i}^{n})\geqslant|{\cal B}_{n}|\times\frac{M\eta}{n},
\]
where $|{\cal B}_{n}|$ denotes the number of elements in ${\cal B}_{n}.$
In particular, $|{\cal B}_{n}|\leqslant\frac{n}{M}.$ It follows that
\begin{align}
K_{n} & =\big(\sum_{i\in{\cal B}_{n}}+\sum_{i\in{\cal G}_{n}}\big)\int_{F_{2}\cap[\tau_{i}^{n},t_{i}^{n}]}\cos(2\phi_{t}^{\lambda}-\alpha_{t})dt\nonumber \\
 & \geqslant-\sum_{i\in{\cal B}_{n}}\mu(F_{2}\cap[\tau_{i}^{n},t_{i}^{n}])+\sum_{i\in{\cal G}_{n}}\int_{F_{2}\cap[\tau_{i}^{n},t_{i}^{n}]}\cos(2\phi_{t}^{\lambda}-\alpha_{t})dt\nonumber \\
 & \geqslant-\frac{L}{n}\times|{\cal B}_{n}|+\sum_{i\in{\cal G}_{n}}\int_{F_{2}\cap[\tau_{i}^{n},t_{i}^{n}]}\cos(2\phi_{t}^{\lambda}-\alpha_{t})dt\nonumber \\
 & \geqslant-\frac{L}{M}+\sum_{i\in{\cal G}_{n}}\int_{F_{2}\cap[\tau_{i}^{n},t_{i}^{n}]}\cos(2\phi_{t}^{\lambda}-\alpha_{t})dt.\label{eq:K_nEst}
\end{align}

We now estimate the last term on the right hand side of (\ref{eq:K_nEst}).
The point is to apply Lemma \ref{lem:DeviPhi} to the context where
\[
[0,L]=[\tau_{i}^{n},t_{i}^{n}],\ [a,b]=[\alpha_{i}^{n}-\varepsilon,\beta_{i}^{n}+\varepsilon],
\]
and 
\[
r=r_{i}^{n}\triangleq2\lambda\cdot\mu\big(\{t:\alpha_{t}\notin[\alpha_{i}^{n}-\varepsilon,\beta_{i}^{n}+\varepsilon]\}\cap[\tau_{i}^{n},t_{i}^{n}]\big).
\]
As one assumption in the lemma, we already have $2\phi_{\tau_{i}^{n}}^{\lambda}\in[\alpha_{i}^{n}-\varepsilon,\beta_{i}^{n}+\varepsilon].$
We must also verify the other standing assumption that 
\begin{equation}
(\beta_{i}^{n}+\varepsilon)-(\alpha_{i}^{n}-\varepsilon)+r_{i}^{n}<\pi.\label{eq:rReq}
\end{equation}
To this end, first note that 
\[
r_{i}^{n}\leqslant2\lambda\mu((F_{1}\backslash F_{2})\cap I_{i}^{n}).
\]
Furthermore, for those $i\in{\cal G}_{n}$, we have 
\[
\mu((F_{1}\backslash F_{2})\cap I_{i}^{n})\leqslant\frac{M\eta}{n}.
\]
As a result, 
\[
r_{i}^{n}\leqslant\frac{2\lambda M\eta}{n}=\frac{2M\eta}{\varepsilon\sin\varepsilon}.
\]
It follows from (\ref{eq:CtyEst}) that
\[
(\beta_{i}^{n}+\varepsilon)-(\alpha_{i}^{n}-\varepsilon)+r_{i}^{n}<3\varepsilon+\frac{2M\eta}{\varepsilon\sin\varepsilon}
\]
for each $i\in{\cal G}_{n}.$ For fixed $\varepsilon$ and $M$, when
$\eta$ is small we can ensure that the condition (\ref{eq:rReq})
is met. We emphasise that such a requirement is legal in view of the
limiting order (\ref{eq:LimOrd}) that will be implemented eventually.
Now we can apply Lemma \ref{lem:DeviPhi} to conclude that 
\[
2\phi_{t}^{\lambda}\in[\alpha_{i}^{n}-\varepsilon-r_{i}^{n},\beta_{i}^{n}+\varepsilon+r_{i}^{n}]\ \ \ \forall t\in[\tau_{i}^{n},t_{i}^{n}].
\]
Consequently, for each $i\in{\cal G}_{n}$ and $t\in F_{2}\cap[\tau_{i}^{n},t_{i}^{n}]$,
we have 
\begin{equation}
|2\phi_{t}^{\lambda}-\alpha_{t}|\leqslant\beta_{i}^{n}-\alpha_{i}^{n}+\varepsilon+r_{i}^{n}<2\varepsilon+\frac{2M\eta}{\varepsilon\sin\varepsilon}.\label{eq:GetBadSetEst}
\end{equation}
For fixed $\varepsilon$ and $M,$ we further require $\eta$ to be
small enough so that $2\varepsilon+\frac{2M\eta}{\varepsilon\sin\varepsilon}<\frac{\pi}{2}$.
As a consequence, we obtain 
\begin{align*}
 & \sum_{i\in{\cal G}_{n}}\int_{F_{2}\cap[\tau_{i}^{n},t_{i}^{n}]}\cos(2\phi_{t}^{\lambda}-\alpha_{t})dt\\
 & \geqslant\cos\big(2\varepsilon+\frac{2M\eta}{\varepsilon\sin\varepsilon}\big)\cdot\sum_{i\in{\cal G}_{n}}\mu(F_{2}\cap[\tau_{i}^{n},t_{i}^{n}])\\
 & =\cos\big(2\varepsilon+\frac{2M\eta}{\varepsilon\sin\varepsilon}\big)\cdot\sum_{i\in{\cal G}_{n}}\big(\mu(F_{2}\cap I_{i}^{n})-\mu(F_{2}\cap[t_{i-1}^{n},\tau_{i}^{n}])\big)\\
 & =\cos\big(2\varepsilon+\frac{2M\eta}{\varepsilon\sin\varepsilon}\big)\cdot\mu(F_{2})-\cos\big(2\varepsilon+\frac{2M\eta}{\varepsilon\sin\varepsilon}\big)\cdot\sum_{i\in{\cal B}_{n}}\mu(F_{2}\cap I_{i}^{n})\\
 & \ \ \ -\cos\big(2\varepsilon+\frac{2M\eta}{\varepsilon\sin\varepsilon}\big)\cdot\sum_{i=1}^{n}\mu(F_{2}\cap[t_{i-1}^{n},\tau_{i}^{n}])\\
 & \geqslant\cos\big(2\varepsilon+\frac{2M\eta}{\varepsilon\sin\varepsilon}\big)\cdot(L-\delta-\eta)-\frac{L}{M}\cos\big(2\varepsilon+\frac{2M\eta}{\varepsilon\sin\varepsilon}\big)\\
 & \ \ \ -\cos\big(2\varepsilon+\frac{2M\eta}{\varepsilon\sin\varepsilon}\big)\cdot\big(\frac{\pi\varepsilon}{2}+\frac{(1+\sin\varepsilon)\eta}{\sin\varepsilon}\big).
\end{align*}
To reach the second term in the last inequality, we have used the
fact that 
\begin{equation}
\sum_{i\in{\cal B}_{n}}\mu(F_{2}\cap I_{i}^{n})\leqslant\frac{L}{n}\times|{\cal B}_{n}|\leqslant\frac{L}{M},\label{eq:BadSum}
\end{equation}
and to reach the third term we have used the estimate (\ref{eq:WaitTime}).

Gathering all the above estimates we have obtained so far, we arrive
at 
\begin{align}
I_{\lambda} & \geqslant-\delta-\eta-\frac{\pi\varepsilon}{2}-\frac{(1+\sin\varepsilon)\eta}{\sin\varepsilon}-\frac{L}{M}+\cos\big(2\varepsilon+\frac{2M\eta}{\varepsilon\sin\varepsilon}\big)\cdot(L-\delta-\eta)\nonumber \\
 & \ \ \ -\frac{L}{M}\cos\big(2\varepsilon+\frac{2M\eta}{\varepsilon\sin\varepsilon}\big)-\cos\big(2\varepsilon+\frac{2M\eta}{\varepsilon\sin\varepsilon}\big)\cdot\big(\frac{\pi\varepsilon}{2}+\frac{(1+\sin\varepsilon)\eta}{\sin\varepsilon}\big).\label{eq:UniformLowerBound}
\end{align}
The proof of Theorem \ref{thm:LConGlo} is thus completed by passing
to the limit in the order specified by (\ref{eq:LimOrd}).

\medskip The following estimate is a direct consequence of the above
argument. It plays an essential role for proving Theorem \ref{thm:Mthm}
in the more general context in the next subsection. We continue to
use the same notation as before and to make all the standing requirements
for the parameters $\delta,\eta,\varepsilon,M.$ However, we do not
take limit for these parameters.
\begin{cor}
\label{cor:BadSetEst}There exists $\Lambda=\Lambda(\delta,\eta,\varepsilon,M),$
such that whenever $\lambda>\Lambda$ and $s\in[0,L]$ satisfies $2\phi_{s}^{\lambda}\in(a,b),$
we have 
\begin{equation}
\mu\big(\big\{ t\in F_{2}\cap[s,L]:|2\phi_{t}^{\lambda}-\alpha_{t}|>2\varepsilon+\frac{2M\eta}{\varepsilon\sin\varepsilon}\big\}\big)\leqslant\frac{\pi\varepsilon}{2}+\frac{1+\sin\varepsilon}{\sin\varepsilon}\eta+\frac{L}{M}.\label{eq:BadSetEst}
\end{equation}
\end{cor}
\begin{proof}
The number $\Lambda$ is chosen so that for any $\lambda>\Lambda$,
we have $\beta_{i}^{n}-\alpha_{i}^{n}<\varepsilon$ (see the discussion
in Section \ref{subsec:S1} leading to the property (\ref{eq:CtyEst})).
Recall that $\varepsilon$ depends on $\delta,$ and $\eta$ is small
depending on $\varepsilon$ and $M$. Hence $\Lambda$ depends on
all these parameters. Suppose that $2\phi_{s}^{\lambda}\in(a,b).$
We treat $2\phi_{s}^{\lambda}$ as the initial condition and restrict
the previous analysis to the interval $[s,L].$ The argument leading
to (\ref{eq:GetBadSetEst}) implies that
\[
\big\{ t\in F_{2}\cap[s,L]:|2\phi_{t}^{\lambda}-\alpha_{t}|>2\varepsilon+\frac{2M\eta}{\varepsilon\sin\varepsilon}\big\}\subseteq\big(\cup_{i\in{\cal B}_{n}}(F_{2}\cap I_{i}^{n})\big)\cup\big(\cup_{i\in{\cal G}_{n}}[t_{i-1}^{n},\tau_{i}^{n}]\big).
\]
The inequality (\ref{eq:BadSetEst}) then follows from (\ref{eq:BadSum})
and (\ref{eq:WaitTime}).
\end{proof}

\subsection{Proof of Theorem \ref{thm:Mthm}: the local case}

We now proceed to develop the proof of Theorem \ref{thm:Mthm} in
general. Suppose that $\gamma:[0,L]\rightarrow\mathbb{R}^{2}$ is
strongly tree-reduced in the sense of Definition \ref{def:SRed}.
Our strategy is to cover the path by small intervals, so that on each
local interval the corresponding estimate (\ref{eq:BadSetEst}) holds.
There is a key missing ingredient in order to patch the estimates
(\ref{eq:BadSetEst}) over different intervals. We must make sure
that the initial condition for $2\phi_{s}^{\lambda}$ on each of the
covering intervals falls in an appropriate region $(a,b)$ to trigger
the relevant estimate. Obtaining such consistency property is non-trivial,
since the initial condition for the current covering is the terminal
condition for the previous covering.

To reduce technical considerations, let us first make one simplification
by assuming that, in Definition \ref{def:SRed} the open interval
$(0,L)$ is replaced by the closed interval $[0,L].$ Namely, we assume
that for each $t\in[0,L],$ there exists a neighbourhood $(u_{t},v_{t})$
of $t$ such that $\gamma|_{[u_{t},v_{t}]\cap[0,L]}$ is a regular
cusp. This simplification allows us to make use of compactness and
finite covers. At the end of this subsection, we discuss how to remove
this restriction (cf. Section \ref{subsec:EndPtIss}). 

\subsubsection{Step one: a covering lemma}

To make the intuition clearer, it is important to choose a nice covering
of $[0,L]$ out of the above assumption. This is the content of the
following lemma.
\begin{lem}
\label{lem:Cov}There exist points $u_{1},\cdots,u_{k-1}$, $v_{1},\cdots,v_{k},$
such that

\vspace{2mm} \noindent (i) $[0,L]=[v_{0},v_{1}]\cup[v_{1},v_{2}]\cup\cdots\cup[v_{k-1},v_{k}]$
where $v_{0}\triangleq0$;\\
(ii) $u_{i}\in(v_{i-1},v_{i})$ for each $1\leqslant i\leqslant k-1$;
\\
(iii) $\alpha|_{[u_{i-1},v_{i}]}$ is a regular cusp for each $1\leqslant i\leqslant k$
where $u_{0}\triangleq0$.
\end{lem}
\begin{proof}
By compactness, we can find a finite family ${\cal A}=\{I_{i}:i=1,\cdots,l\}$
of distinct intervals that cover $[0,L]$, where each interval $I_{i}$
is relatively open in $[0,L]$ and $\alpha|_{\overline{I_{i}}\cap[0,L]}$
is a regular cusp for each $i$. The point $t=0$ is covered by some
member in ${\cal A},$ say $I_{1}=(0,v_{1}).$ If $v_{1}\geqslant L$,
we set $v_{1}\triangleq L$ and we are done. Otherwise, the point
$t=v_{1}$ is covered by some member in ${\cal A}\backslash\{I_{1}\},$
say $I_{2}=(u'_{1},v_{2}).$ If $v_{2}\geqslant L,$ we are done by
setting $v_{2}\triangleq L$ and choosing any point $u_{1}\in(u_{1}',v_{1}).$
If $v_{2}<L$, we continue the process. Inductively, $v_{i}$ is covered
by some member in ${\cal A}\backslash\{I_{1},\cdots,I_{i}\}$, say
$I_{i+1}=(u_{i}',v_{i+1}).$ We choose $u_{i}\in(\max\{u_{i}',v_{i-1}\},v_{i})$
and proceed further. The process terminates after finitely many steps
since ${\cal A}$ is finite.
\end{proof}
The figure below illustrates the covering specified by Lemma \ref{lem:Cov}
when $k=4.$ \begin{figure}[H]
\begin{center}  
\includegraphics[scale=0.32]{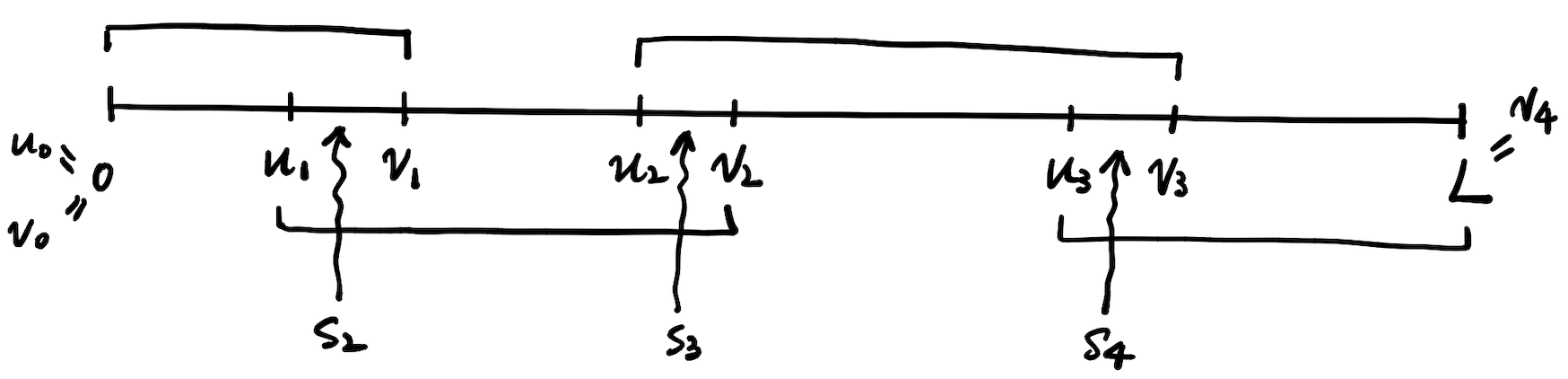} 
\protect\caption{The covering structure when $k=4$.}\label{fig:Cover} 
\end{center} 
\end{figure}\noindent In what follows, we always work with a fixed covering structure
given by Lemma \ref{lem:Cov}.

\subsubsection{Step two: consistency of initial conditions}

From Lemma \ref{lem:Cov}, we know that $\alpha|_{[u_{i-1},v_{i}]}$
is a regular cusp. In particular, we know by assumption that 
\[
\alpha_{t}\in[a_{i},a_{i}+\pi]\ \ \ {\rm for\ a.a.\ }t\in[u_{i-1},v_{i}]
\]
with some $a_{i}\in\mathbb{R}$. The main issue here is that, we cannot
directly apply the results from Section \ref{subsec:Glo}, since we
do not know whether $2\phi_{u_{i-1}}^{\lambda}\in(a_{i},a_{i}+\pi).$
Such a requirement on the initial condition is critical in the previous
argument. Nonetheless, the following lemma tells us that we can find
$s_{i}\in[u_{i-1},v_{i-1})$ (which may depend on $\lambda$) such
that $2\phi_{s_{i}}^{\lambda}\in(a_{i},a_{i}+\pi)$. As a result,
$s_{i}$ can be treated as the initial time over the portion of $[s_{i},v_{i}]$.
\begin{lem}
There exists $\Lambda>0,$ such that for any $\lambda>\Lambda$ and
$1\leqslant i\leqslant k-1,$ we can always find $s_{i}\in[u_{i-1},v_{i-1})$
satisfying 
\[
2\phi_{s_{i}}^{\lambda}\in(a_{i},a_{i}+\pi).
\]
\end{lem}
\begin{proof}
We continue to use the notation in Section \ref{subsec:Glo} but applied
to the context of $\alpha|_{[u_{i-1},v_{i}]}$ for each $i$. Recall
that $\delta,\varepsilon,\eta,M$ be given parameters. We are then
able to define two compact subsets $F_{i}\supseteq F_{i}'$ of $[u_{i-1},v_{i}]$
playing the roles of $F_{1},F_{2}$, and two numbers $a_{i}'>a_{i},$
$b_{i}'<b_{i}\triangleq a_{i}+\pi$ playing the roles of $a_{\delta},b_{\delta}$
in that section. We further require that these parameters satisfy
the following constraints: 
\begin{equation}
2\varepsilon+\frac{2\eta M}{\varepsilon\sin\varepsilon}<\min_{1\leqslant i\leqslant k}\min\{a_{i}'-a_{i},b_{i}-b_{i}'\},\label{eq:1Choice}
\end{equation}
and
\begin{equation}
\big(\frac{\pi\varepsilon}{2}+\frac{1+\sin\varepsilon}{\sin\varepsilon}\eta+\frac{L}{M}\big)+2(\delta+\eta)<\min_{1\leqslant i\leqslant k-1}(v_{i}-u_{i}).\label{eq:2Choice}
\end{equation}
More quantitatively, we first choose $\delta,\varepsilon$ to be small,
then $M$ to be large, and finally $\eta$ to be small. This is consistent
with the limiting order (\ref{eq:LimOrd}). Now we are in a position
to apply the quantitative estimate given by Corollary \ref{cor:BadSetEst}
to each portion $\alpha|_{[u_{i-1},v_{i}]}$. Note that we do not
take limits for the parameters $\delta,\varepsilon,\eta,M$ here.
The constant $\Lambda$ appearing in Corollary \ref{cor:BadSetEst}
depends on these parameters as well as on the fixed covering structure
given by the $[u_{i-1},v_{i}]$'s.

For any given $\lambda>\Lambda,$ we are going to choose $s_{i}\in[u_{i-1},v_{i-1})$
inductively on $i$, such that $2\phi_{s_{i}}^{\lambda}\in(a_{i},b_{i})$.
We start by fixing $2\phi_{0}^{\lambda}\in(a_{1},b_{1})$ and choosing
$s_{1}\triangleq0.$ Suppose that $s_{i}\in[u_{i-1},v_{i-1})$ is
already selected with the desired property, and we want to define
$s_{i+1}$ properly. According to Corollary \ref{cor:BadSetEst} and
the requirement (\ref{eq:2Choice}), we have 
\[
\mu\big(\big\{ t\in F_{i}'\cap[s_{i},v_{i}]:\big|2\phi_{t}^{\lambda}-\alpha_{t}\big|>2\varepsilon+\frac{2\eta M}{\varepsilon\sin\varepsilon}\big\}\big)<v_{i}-u_{i}-2\delta-2\eta.
\]
Since $\mu((F_{i}')^{c})<\delta+\eta$ from Section \ref{subsec:Glo},
it follows that 
\begin{align*}
 & \mu\big(\big\{ t\in F_{i}'\cap[s_{i},v_{i}]:\big|2\phi_{t}^{\lambda}-\alpha_{t}\big|\leqslant2\varepsilon+\frac{2\eta M}{\varepsilon\sin\varepsilon}\big\}\big)\\
 & >\mu\big(F_{i}'\cap[s_{i},v_{i}]\big)-(v_{i}-u_{i}-2\delta-2\eta)\\
 & \geqslant(v_{i}-s_{i})-(\delta+\eta)-(v_{i}-u_{i}-2\delta-2\eta)\\
 & =u_{i}-s_{i}+\delta+\eta.
\end{align*}
As a result, we have
\[
\mu(C_{i}\cap[u_{i},v_{i}])>\delta+\eta,
\]
where
\[
C_{i}\triangleq\big\{ t\in[s_{i},v_{i}]:\big|2\phi_{t}^{\lambda}-\alpha_{t}\big|\leqslant2\varepsilon+\frac{2\eta M}{\varepsilon\sin\varepsilon}\big\}.
\]
Since $\mu((F_{i+1}')^{c})<\delta+\eta$, we conclude that 
\[
C_{i}\cap F_{i+1}'\cap[u_{i},v_{i})\neq\emptyset.
\]
Pick any point in the above set and define it as $s_{i+1}.$ Note
that from Section \ref{subsec:Glo} we also have $\alpha_{s_{i+1}}\in[a_{i+1}',b_{i+1}']$
(since $s_{i+1}\in F_{i+1}'$). Therefore, the requirement (\ref{eq:1Choice})
further implies that $2\phi_{s_{i+1}}^{\lambda}\in(a_{i+1},b_{i+1}).$
This gives the desired construction of $s_{i+1}.$
\end{proof}

\subsubsection{Step three: patching up the estimates}

We now proceed to establish the global lower estimate. We continue
to work in the previous set-up. For each given $\lambda>\Lambda,$
the previous choice of $s_{i}$ allows us to apply the estimate (\ref{eq:BadSetEst})
to $\alpha|_{[s_{i},v_{i}]}$. In particular, for each $i$ we have
\begin{align*}
 & \mu\big(\big\{ t\in F_{i}'\cap[s_{i},v_{i}]:\big|2\phi_{t}^{\lambda}-\alpha_{t}\big|>2\varepsilon+\frac{2\eta M}{\varepsilon\sin\varepsilon}\big\}\big)\leqslant\frac{\pi\varepsilon}{2}+\frac{1+\sin\varepsilon}{\sin\varepsilon}\eta+\frac{L}{M}.
\end{align*}
Let us define 
\[
D\triangleq\big\{ t\in[0,L]:\left|2\phi_{t}^{\lambda}-\alpha_{t}\right|>2\varepsilon+\frac{2\eta M}{\varepsilon\sin\varepsilon}\big\}.
\]
Then we have

\begin{align}
\mu(D) & \leqslant\sum_{i=1}^{k}\mu\big(\big\{ t\in[s_{i},v_{i}]:\left|2\phi_{t}^{\lambda}-\alpha_{t}\right|>2\varepsilon+\frac{2\eta M}{\varepsilon\sin\varepsilon}\big\}\big)\nonumber \\
 & \leqslant k\big(\frac{\pi\varepsilon}{2}+\frac{1+\sin\varepsilon}{\sin\varepsilon}\eta+\frac{2L}{M}+\eta+\delta\big),\label{eq:DEst}
\end{align}
and 
\begin{align*}
\int_{0}^{L}\cos(\alpha_{t}-2\phi_{t}^{\lambda})dt & \geqslant-\mu(D)+\cos\big(2\varepsilon+\frac{2\eta M}{\varepsilon\sin\varepsilon}\big)\mu(D^{c})\\
 & \geqslant L\cos\big(2\varepsilon+\frac{2\eta M}{\varepsilon\sin\varepsilon}\big)-\big(1+\cos\big(2\varepsilon+\frac{2\eta M}{\varepsilon\sin\varepsilon}\big)\big)\mu(D).
\end{align*}
By substituting the estimate (\ref{eq:DEst}) and taking limit in
the order (\ref{eq:LimOrd}), we conclude that
\[
\underset{\lambda\rightarrow\infty}{\overline{\lim}}\int_{0}^{L}\cos(\alpha_{t}-2\phi_{t}^{\lambda})dt\geqslant L.
\]

\subsubsection{\label{subsec:EndPtIss}Step four: removing the assumption at the
end points}

Finally, we come back to relax the requirement on the endpoints $t=0,L.$
More precisely, we now assume that for each $t\in(0,L)$ (not including
the endpoints), there is a neighbourhood $(u_{t},v_{t})$ of $t$
on which $\gamma$ is a regular cusp. Having all the previous analysis
at hand, dealing with this situation only requires minor technical
effort.

To elaborate this, let $\kappa>0$ be a given number. Then we can
write 
\[
\int_{0}^{L}\cos(\alpha_{t}-2\phi_{t}^{\lambda})dt\geqslant-2\kappa+\int_{\kappa}^{L-\kappa}\cos(\alpha_{t}-2\phi_{t}^{\lambda})dt.
\]
On the other hand, we know that $\gamma|_{[\kappa,L-\kappa]}$ satisfies
Definition \ref{def:SRed} up to the end points $\kappa$ and $L-\kappa$.
In order to apply the previous results to $\gamma|_{[\kappa,L-\kappa]}$,
the only requirement is a suitable initial condition for $2\phi_{\kappa}^{\lambda}.$
But we know that (cf. Lemma \ref{eq:ODEw})
\[
\rho_{\kappa}^{\lambda}e^{i\phi_{\kappa}^{\lambda}}=w_{\kappa}^{\lambda}=\Gamma_{L-\kappa,L}^{\lambda}\xi^{\lambda}.
\]
Since $\Gamma_{L-\kappa,L}^{\lambda}$ is invertible, by choosing
$\xi^{\lambda}$ properly we can certainly guarantee that $2\phi_{\kappa}^{\lambda}$
satisfies a desired condition (i.e. $2\phi_{\kappa}^{\lambda}\in(a_{1},a_{1}+\pi)$
using the notation from the previous discussion). As a result, we
conclude that 
\[
\underset{\lambda\rightarrow\infty}{\overline{\lim}}\int_{0}^{L}\cos(\alpha_{t}-2\phi_{t}^{\lambda})dt\geqslant-2\kappa+(L-2\kappa)=L-4\kappa.
\]
By letting $\kappa\rightarrow0^{+},$ we obtain the desired estimate.

Up to this point, the proof of Theorem \ref{thm:Mthm} is complete.

\section{\label{sec:MCusp}An extension: singular cusps}

We have mentioned at the beginning that there is another type of cusps
which is more singular in terms of detecting the tree-reduced property
and is thus harder to deal with. In this section, we discuss how the
previous analysis, when combined with a suitable comparison lemma,
can be adapted to treat this more singular case. For the sake of conciseness
and for conveying the essential idea better, we only consider a typical
example instead of trying to write down an abstract condition capturing
such type of cusps. We remark at the end of this section on how the
argument can be adapted to a more general situation.

We first illustrate this singular type of cusps in the figure below.\begin{figure}[H]  
\begin{center}  
\includegraphics[scale=0.28]{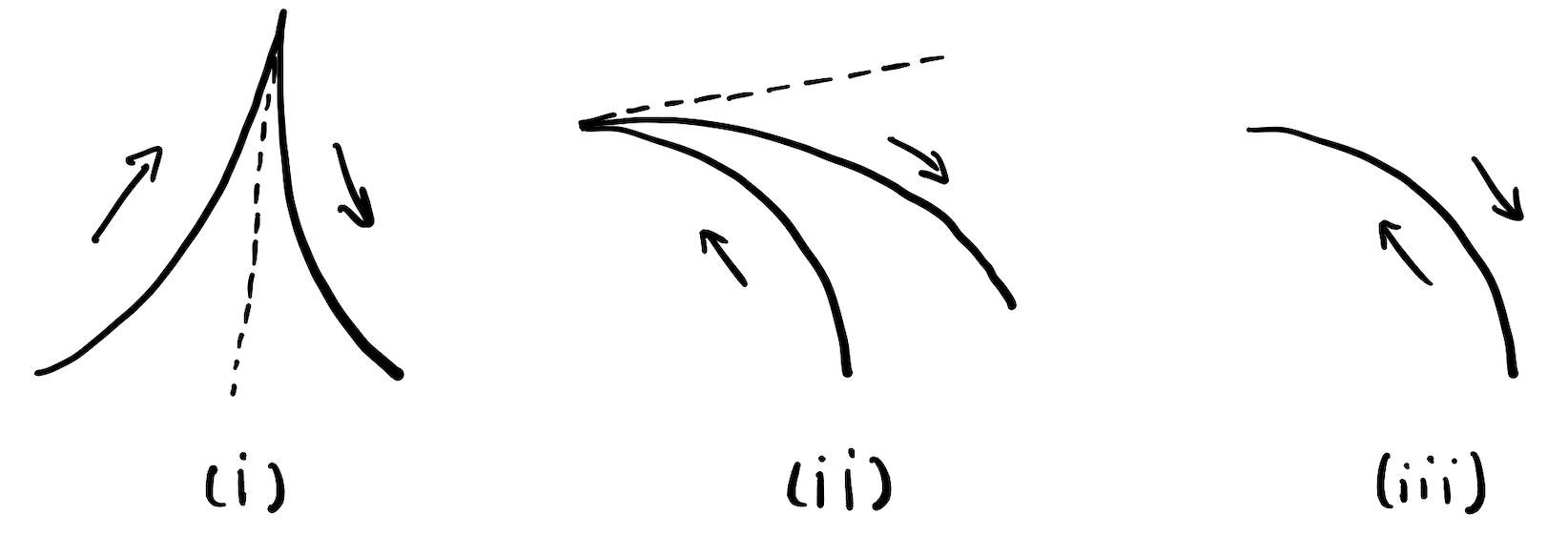} 
\protect\caption{Regular cusp, Singular cusp and tree-like cusp.}\label{fig:SCusp} 
\end{center} 
\end{figure} \noindent The leftmost path represents a typical regular cusp in
the sense of Definition\ref{def:Cusp}. The middle path represents
the singular cusp that we are considering here. The rightmost path
represents a tree-like cusp that has trivial signature. Note that
these paths are all local, i.e. we are only zooming in the part near
the cusp singularity.

It is not hard to describe why these two types of cusps are different
in terms of the ``degree of tree-reducedness''. For the moment let
us assume that the paths are $C^{2}$ near the cusp singularity point.
For the regular cusp in Figure \ref{fig:SCusp} (i), one detects its
tree-reducedness directly from the fact that the second derivative
of the path (more precisely, the first derivative of the angular path
$\beta_{t}$) does not change sign when passing through the singularity.
This also accounts for the property that $\beta_{t}$ takes values
in an interval of length strictly less than $\pi$ after removing
an arbitrarily small neighbourhood of the singularity (cf. Definition
\ref{def:Cusp}). However, for the singular cusp in Figure \ref{fig:SCusp}
(ii), one \textit{cannot} detect whether it is tree-reduced or not
by examining the sign of the second derivative. Instead, the way to
distinguish it from the tree-like cusp in Figure \ref{fig:SCusp}
(iii) is through looking at the precise magnitude of the second derivative
on both sides of the singularity. To put it in another way, it is
the speed of change for $\beta_{t}$ rather than the direction of
change that distinguishes it from being tree-like. Such information
is finer than what is contained in Definition \ref{def:Cusp}. As
a result, dealing with this case requires more delicate analysis.

In what follows, we consider one typical example of singular cusps.
To be more convenient for using the equation (\ref{eq:ODEw}), we
directly specify $\alpha_{t}$ instead of $\beta_{t}$ (recall that
they are related by $\alpha_{t}=\beta_{L-t}$). Let $L,r>0$ and $a\in\mathbb{R}$
be given fixed. Consider a given \textit{continuous}, \textit{strictly
increasing} function 
\[
\theta:[0,L/2]\rightarrow[a-r,a]
\]
with $\theta_{0}=a-r$ and $\theta_{L/2}=a$. The simplest example
of $\theta$ is the linear function. We define the angular path $\alpha:[0,L]\rightarrow\mathbb{R}$
by 
\begin{equation}
\alpha_{t}\triangleq\begin{cases}
\theta_{t}, & t\in[0,L/2);\\
c\cdot(\theta_{L-t}-a)+(a-\pi), & t\in(L/2,L],
\end{cases}\label{eq:CuspAng}
\end{equation}
where $c>0$ is a given fixed number. Note that $\alpha_{L/2-}=a$
and $\alpha_{L/2+}=a-\pi.$ The resulting path $\gamma_{t}$ is given
by the equation (\ref{eq:DefPath}) where $\beta_{t}\triangleq\alpha_{L-t}.$

When $c=1,$ $\gamma$ is tree-like. When $c\neq1,$ $\gamma$ has
the shape of Figure \ref{fig:SCusp} (ii). By considering the reversal
path if necessary, we may assume without loss of generality that $0<c<1.$
In addition, since only the local behaviour of $\gamma$ near the
singularity is relevant, we may also assume that $r$ is small, say
$r<\frac{\pi}{4}.$

Our main result of this section is the following.
\begin{thm}
\label{thm:MCusp}The signature asymptotics formula holds for the
path $\gamma$ defined as above.
\end{thm}
To prove this result, we first state a lemma which also directly solves
the case when $\alpha:[0,L]\rightarrow\mathbb{R}$ is a continuous
function (i.e. the ${\cal C}^{1}$-case). We defer its proof to the
appendix so as not to distract the reader from the main discussion.
In the ${\cal C}^{1}$-case, Lyons-Xu \cite{LX15} also had a similar
result for the equations of the hyperbolic development.
\begin{lem}
\label{lem:C1Lem}Let $\alpha:[0,L]\rightarrow\mathbb{R}$ be a continuous
function. Recall that the function $\phi_{t}^{\lambda}$ is defined
by the angular equation \eqref{eq:AEqn}. Suppose that there exists
$\kappa\in(0,\pi)$ such that 
\[
|2\phi_{0}^{\lambda}-\alpha_{0}|\leqslant\kappa\ \ \ \text{for all }\lambda>0.
\]
Then for any $t_{0}>0,$ $2\phi_{t}^{\lambda}$ converges uniformly
to $\alpha_{t}$ on $[t_{0},L]$ as $\lambda\rightarrow\infty$. In
addition, if $2\phi_{0}^{\lambda}=\alpha_{0}$ for all $\lambda,$
then the uniform convergence holds on $[0,L].$
\end{lem}
The key point for proving Theorem \ref{thm:MCusp} is to compare the
angle dynamics $2\phi_{t}^{\lambda}$ associated with the cusp path
$\alpha_{t}$ to the one corresponding to the tree-like case. Let
us begin with the trivial observation that the angular equations \eqref{eq:AEqn}
on $[0,L/2]$ for the two cases ($c\neq1$ vs $c=1$) are identical.
In addition, if we take $2\phi_{0}^{\lambda}=\alpha_{0}$, according
to Lemma \ref{lem:C1Lem} and Lemma \ref{lem:RangePhi}, we have
\[
2\phi_{t}^{\lambda}\in R_{1}\triangleq[a-r,a]\ \ \ \forall t\in[0,L/2],
\]
and $\|2\phi_{\cdot}^{\lambda}-\alpha_{\cdot}\|_{\infty;[0,L/2]}$
can be made arbitrarily small when $\lambda$ is large.

To understand the portion of $[L/2,L]$, we first look at the tree-like
situation $c=1$. Let us use $\psi_{t}^{\lambda}$ to denote the corresponding
solution to the angular equation \eqref{eq:AEqn} on $[L/2,L]$ for
this case. By the definition (\ref{eq:CuspAng}) with $c=1$, we know
that 
\begin{align}
\frac{d}{dt}\psi_{t}^{\lambda} & =-\lambda\sin(2\psi_{t}^{\lambda}-(\theta(L-t)-\pi))\nonumber \\
 & =\lambda\sin(2\psi_{t}^{\lambda}-\theta(L-t))\ \ \ \forall t\in[L/2,L].\label{eq:AEtrlk}
\end{align}
As a result of uniqueness, we have
\[
\psi_{t}^{\lambda}=\psi_{L-t}^{\lambda}=\phi_{L-t}^{\lambda}\ \ \ \forall t\in[L/2,L].
\]
In other words, on the second half $[L/2,L]$, the path $\psi_{t}^{\lambda}$
is just the reversal of the first half $[0,L/2]$.

Now we return to the cusp situation with $0<c<1$ given fixed. The
angular equation \eqref{eq:AEqn} can be rewritten as
\begin{equation}
\frac{d\phi_{t}^{\lambda}}{dt}=\lambda\sin\big(2\phi_{t}^{\lambda}-\theta_{L-t}-\varepsilon\cdot(a-\theta_{L-t})\big)\ \ \ \forall t\in[L/2,L],\label{eq:AEcusp}
\end{equation}
where $\varepsilon\triangleq1-c.$ Recall that $\psi_{t}^{\lambda}\triangleq\phi_{L-t}^{\lambda}$
($t\in[L/2,L]$) gives the angular solution in the tree-like case
on $[L/2,L].$

The following comparison lemma is the key step towards the proof of
Theorem \ref{thm:MCusp}.
\begin{lem}
\label{lem:Comp}For each $\lambda>0,$ we have 
\[
\phi_{t}^{\lambda}\leqslant\psi_{t}^{\lambda}\ \ \ \forall t\in[L/2,L].
\]
\end{lem}
\begin{proof}
Let $\zeta_{t}\triangleq\psi_{t}^{\lambda}-\phi_{t}^{\lambda}$ ($t\in[L/2,L]$).
Then $\zeta_{L/2}=0$, and using the equations (\ref{eq:AEtrlk}),
(\ref{eq:AEcusp}) for the two functions, we see that $\zeta_{t}$
satisfies the equation
\begin{align*}
\frac{d\zeta_{t}}{dt} & =2\lambda\cos\frac{(2\psi_{t}^{\lambda}-\theta_{L-t})+(2\phi_{t}^{\lambda}-\theta_{L-t})-\varepsilon(a-\theta_{L-t})}{2}\\
 & \ \ \ \times\sin\big(\zeta_{t}+\frac{\varepsilon(a-\theta_{L-t})}{2}\big).
\end{align*}

Firstly, we claim that $\zeta_{t}\geqslant0$ when $t$ is small.
To this end, we set $\eta_{t}\triangleq\frac{\varepsilon(a-\theta_{L-t})}{2}.$
By the assumption on $\theta_{t},$ the function $\eta_{t}$ is non-negative,
and strictly increases from $0$ to $\frac{\varepsilon r}{2}$. Therefore,
we have 
\begin{align*}
|\zeta_{t}| & \leqslant2\lambda\int_{L/2}^{t}\big|\sin(\zeta_{s}+\eta_{s})\big|ds\\
 & \leqslant2\lambda\int_{L/2}^{t}|\zeta_{s}|ds+2\lambda\big(t-\frac{L}{2}\big)\eta_{t}.
\end{align*}
It follows from Gr\"onwall's inequality that 
\[
|\zeta_{t}|\leqslant2\lambda\big(t-\frac{L}{2}\big)\eta_{t}\times e^{2\lambda(t-L/2)},\ \ \ t\in[L/2,L].
\]
As a consequence, we have 
\[
\zeta_{t}+\eta_{t}\geqslant\big(1-2\lambda\big(t-\frac{L}{2}\big)e^{2\lambda(t-L/2)}\big)\eta_{t}.
\]
In particular, there exists $\tau_{1}\in(L/2,L)$ (depending on $\lambda$)
such that 
\[
\zeta_{t}+\eta_{t}\geqslant\frac{1}{2}\eta_{t}\geqslant0\ \ \ \forall t\in[L/2,\tau_{1}].
\]
We also require that $\zeta_{t}+\eta_{t}\leqslant\pi$ by shrinking
$\tau_{1}$ if necessary. As a result, we have
\[
\sin(\zeta_{t}+\eta_{t})\geqslant0\ \ \ \forall t\in[L/2,\tau_{1}].
\]
On the other hand, observe that 
\[
|2\psi_{t}^{\lambda}-\theta_{L-t}|\leqslant r
\]
since they both stay in the region $R_{1}.$ By the relation $\phi_{L/2}^{\lambda}=\psi_{L/2}^{\lambda}$
and the continuity of $\phi_{t}^{\lambda}$ at $t=L/2$, there exists
$\tau_{2}\in(L/2,L)$ (depending on $\lambda$) satisfying 
\[
\cos\frac{(2\psi_{t}^{\lambda}-\theta_{L-t})+(2\phi_{t}^{\lambda}-\theta_{L-t})-\varepsilon(a-\theta_{L-t})}{2}\geqslant\cos\frac{(3+\varepsilon)r}{2}=:\kappa_{r}>0
\]
for $t\in[L/2,\tau_{2}]$ (recall we have presumed that $c\in(0,1)$
and $r<\frac{\pi}{4}$). By taking $\tau\triangleq\tau_{1}\wedge\tau_{2}$,
we obtain
\[
\zeta_{t}\geqslant2\lambda\kappa_{r}\int_{L/2}^{t}\sin(\zeta_{s}+\eta_{s})ds\geqslant0\ \ \ \forall t\in[L/2,\tau].
\]

Next, we claim that $\zeta_{t}\geqslant0$ for all $t\in[L/2,L].$
Suppose on the contrary that, there is some $t$ such that $\zeta_{t}<0.$
This $t$ must be in the interval $(\tau,L].$ A standard argument
allows us to find $t_{1},t_{2}\in[\tau,L]$ such that $\zeta_{t_{1}}=0$
and $\zeta_{t}<0$ for $t\in(t_{1},t_{2}]$. However, since $\eta_{t}$
is strictly increasing, we know that $\eta_{t_{1}}\in(0,\varepsilon r/2).$
As a result, 
\begin{align*}
\zeta_{t_{1}}' & =2\lambda\cos\big(2\psi_{t_{1}}^{\lambda}-\theta_{L-t_{1}}-\frac{\varepsilon(a-\theta_{L-t_{1}})}{2}\big)\sin\eta_{t_{1}}\\
 & \geqslant2\lambda\cos\big(\frac{(2+\varepsilon)r}{2}\big)\sin\eta_{t_{1}}\\
 & >0,
\end{align*}
which is clearly a contradiction. Therefore, $\zeta_{t}\geqslant0$
on $[L/2,L].$
\end{proof}
Now we are in a position to give the proof of Theorem \ref{thm:MCusp}.
The idea is that, at any fixed time $t^{*}>L/2$, when $\lambda$
is large $2\phi_{t^{*}}^{\lambda}$ gets pushed into the interval
$R_{2}\triangleq[a-\pi-cr,a-\pi]$ (the range of $\alpha$ on $[L/2,L]$).
As a consequence, we can then apply results obtained in Section \ref{sec:PfMthm}
to the portion of $[t^{*},L].$ The theorem then follows as $t^{*}-L/2$
can be made arbitrarily small.

\begin{proof}[Proof of Theorem \ref{thm:MCusp}]

Let $t_{1}<t_{2}$ be two fixed times in $(L/2,L).$ We claim that,
there exists $\Lambda>0$ such that 
\begin{equation}
2\phi_{t_{2}}^{\lambda}\in R_{2}\ \ \ \forall\lambda>\Lambda.\label{eq:PushCl}
\end{equation}
If this is true, the argument developed in Section \ref{sec:PfMthm}
applied to the portion of $[t_{2},L]$ implies that
\begin{align*}
\underset{\lambda\rightarrow\infty}{\overline{\lim}}\int_{0}^{L}\cos(\alpha_{t}-2\phi_{t}^{\lambda})dt & =\underset{\lambda\rightarrow\infty}{\overline{\lim}}\big(\int_{0}^{L/2}+\int_{L/2}^{t_{2}}+\int_{t_{2}}^{L}\big)\cos(\alpha_{t}-2\phi_{t}^{\lambda})dt\\
 & \geqslant\frac{L}{2}+(L-t_{2})-\big(\frac{L}{2}-t_{2}\big).
\end{align*}
The desired lower bound follows by letting $t_{2}\rightarrow L/2$.

The proof of the claim (\ref{eq:PushCl}) contains the following three
observations.

\vspace{2mm} \noindent (i) From Lemma \ref{lem:Comp}, we know that
\[
2\phi_{t_{1}}^{\lambda}\leqslant2\psi_{t_{1}}^{\lambda}=2\phi_{L-t_{1}}^{\lambda}.
\]
(ii) According to Lemma \ref{lem:C1Lem}, $2\psi_{t_{1}}^{\lambda}$
and $\theta_{L-t_{1}}$ can be made arbitrarily close when $\lambda$
is large. \\
(iii) The distance between $\alpha_{t_{1}}$ and $\theta_{L-t_{1}}$
is strictly less than $\pi.$ Indeed,
\begin{align*}
\theta_{L-t_{1}}-\alpha_{t_{1}} & =\theta_{L-t_{1}}-c\cdot(\theta_{L-t_{1}}-a)-a+\pi\\
 & =\pi-(1-c)(a-\theta_{L-t_{1}}),
\end{align*}
which is less than $\pi$ since $c\in(0,1)$ and $\theta_{L-t_{1}}<a$.

\vspace{2mm} \noindent  To prove the claim (\ref{eq:PushCl}) precisely,
first observe from Lemma \ref{lem:RangePhi} that, if $2\phi_{t}^{\lambda}$
ever enters the region $R_{2}$ during $(L/2,t_{2})$, it will remain
in $R_{2}$ afterwards since $\alpha_{t}\in R_{2}$ for $t\in[L/2,L]$.
In particular, we have $2\phi_{t_{2}}^{\lambda}\in R_{2}$ in this
case. Let us now assume the other case that $2\phi_{t}^{\lambda}\in[a-\pi,a]$
for $t\in[L/2,t_{2}]$. According to the above points (i)--(iii),
in this case we see that the distance between $2\phi_{t_{1}}^{\lambda}$
and $\alpha_{t_{1}}$ is uniformly less than $\pi$ for all large
$\lambda$. As a consequence of Lemma \ref{lem:C1Lem}, we conclude
that 
\[
\lim_{\lambda\rightarrow\infty}\big|2\phi_{t_{2}}^{\lambda}-\alpha_{t_{2}}\big|=0.
\]
In particular, $2\phi_{t_{2}}^{\lambda}\in R_{2}$ when $\lambda$
is large. This proves the desired claim.

\end{proof}

We give some further comments to conclude the discussion for this
section. Although we are only considering a particular type of examples
here, the above argument can be adapted to deal with the more general
situation where $\alpha_{t}$ is $C^{2}$ near the singularity $t_{*}\in(0,L)$
and $|\alpha'(t_{*}-)|\neq|\alpha'(t_{*}+)|$. For simplicity, suppose
that on the portion of $[0,t_{*}]$ we are in the setting of Section
\ref{sec:PfMthm}, so that we have the estimate 
\begin{equation}
\int_{0}^{t_{*}}\cos(\alpha_{t}-2\phi_{t}^{\lambda})dt\gtrsim t_{*}\label{eq:BefSig}
\end{equation}
when $\lambda$ is large. To deal with the portion after $t_{*}$,
the point is that, the formula (\ref{eq:CuspAng}) provides a good
approximation of the actual path $\alpha_{t}$ in a small neighbourhood
$(t_{*}-h,t_{*}+h)$ of $t_{*}$, where the parameter $c\neq1$ captures
the difference between the magnitudes of the left and right derivatives
of $\alpha$ at the singularity $t_{*}.$ Using a suitable comparison
lemma, one can then show that, after passing through the singularity
$t_{*}$, when $\lambda$ is large the angular path $2\phi_{t}^{\lambda}$
gets pushed into the region where $\alpha|_{(t_{*},t_{*}+h)}$ belongs.
As a result, the initial condition of $2\phi_{t}^{\lambda}$ on the
portion of $[t_{*}+h,L]$ is favorable (relative to $\alpha(t_{*}+h)$),
and the analysis developed in Section \ref{sec:PfMthm} again leads
to the estimate 
\begin{equation}
\int_{t_{*}+h}^{L}\cos(\alpha_{t}-2\phi_{t}^{\lambda})dt\gtrsim L-t_{*}-h\label{eq:AftSig}
\end{equation}
when $\lambda$ is large. By adding up the two estimates (\ref{eq:BefSig}),
(\ref{eq:AftSig}) and letting $h\rightarrow0^{+},$ we arrive at
the formula (\ref{eq:LengConj}).

\section{\label{sec:FurQ}Some further questions}

One may wonder how much further one needs to go towards a complete
solution to the length conjecture (\ref{eq:LengConj})? In our modest
opinion, a critical step is to identify a suitable way to capture
the \textit{degree} of tree-reducedness at a \textit{quantitative}
level, which is at the same time convenient to be utilised in the
signature analysis. This is a quite challenging part of the problem
as the tree-reduced property is essentially topological. Although
Theorem \ref{thm:Mthm} does not resolve the problem completely, we
believe that it provides an interesting and important attempt along
this philosophy. For instance, in the simplified context of Example
\ref{exa:NoCusp}, the condition that $\beta_{t}$ stays in an interval
of length $r<\pi$ reflects the extent to which the path cannot turn
around immediately (thus being tree-reducedness). As $r$ gets smaller,
the path tends to be ``more tree-reduced''. While as $r$ gets closer
to $\pi,$ the condition becomes less sensitive for detecting the
tree-reduced property, and the extreme case of $r=\pi$ allows the
possibility of creating tree-like pieces.

To go deeper into the study, among others there are at least two interesting
questions one can raise and investigate.
\begin{question}
Can we extend the current approach to higher dimensional paths?
\end{question}
At the moment, this is not entirely straight forward. For higher dimensional
paths, one may resort to Cartan developments onto higher dimensional
Lie groups, e.g. ${\rm SL}_{n}(\mathbb{R}),$ ${\rm SO}(p,q)$ etc.
If one designs the development in a clever way, it may not be too
surprising to end up with an ODE system in which the angular component
for the group action is decoupled from the radial component just like
the equations \eqref{eq:REqn} and \eqref{eq:AEqn}. However, one faces
another level of challenge (which is harder to overcome) due to the
lack of monotonicity properties for the angular equation, since the
angle dynamics is now taking values in the $n$-sphere $S^{n}$ rather
than in $\mathbb{R}$ (or $S^{1}$).

On the other hand, we have made use of the intuition that the tree-reduced
property is reflected by the incapability of making a $\pi$-turn
locally. There is a weaker type of conditions that captures this property
in a more direct way, which is expressed in terms of turning angles:
\begin{equation}
|\beta_{t}-\beta_{s}|\leqslant\kappa\ \ \ {\rm for\ a.a.\ }s,t\label{eq:TurnCond}
\end{equation}
with some given constant $\kappa\in(0,\pi).$
\begin{question}
Is it possible to prove the signature asymptotics formula (\ref{eq:LengConj})
under the condition (\ref{eq:TurnCond}) or more generally under a
suitably localised version of (\ref{eq:TurnCond})?
\end{question}
Clearly, the condition (\ref{eq:GloCond}) implies the condition (\ref{eq:TurnCond}).
However, it is possible to construct a bounded variation path that
satisfies (\ref{eq:TurnCond}) but is not strongly tree-reduced in
the sense of Definition \ref{def:SRed}. For instance, let $\{A,B,C\}$
be a Lebesgue measurable partition of $[0,1],$ such that for every
open subset $U$ of $[0,1]$ one has
\[
\mu(U\cap A)>0,\ \mu(U\cap B)>0,\ \mu(U\cap C)>0.
\]
The existence of $\{A,B,C\}$ is a (non-trivial) exercise in real
analysis. Define $\beta:[0,1]\rightarrow S^{1}$ by 
\[
\beta_{t}=0\cdot{\bf 1}_{A}(t)+e^{2\pi i/3}\cdot{\bf 1}_{B}(t)+e^{4\pi i/3}\cdot{\bf 1}_{C}(t).
\]
Then $\beta_{t}$ satisfies (\ref{eq:TurnCond}) with $\kappa=\frac{2\pi}{3}$,
where the distance $|\beta_{t}-\beta_{s}|$ is understood as the $S^{1}$-distance.
The resulting path $\gamma_{t}\triangleq\int_{0}^{t}\beta_{s}ds$
is tree-reduced but not strongly tree-reduced. At this point, it is
not so clear if the current strategy can be adapted to deal with this
weaker type of conditions.

\begin{appendices}
\section{Proof of Lemma \ref{lem:C1Lem} and the $\mathcal{C}^1$-case.}

In this section, we give the proof of Lemma \ref{lem:C1Lem}, which
also directly implies the signature asymptotics formula (\ref{eq:LengConj})
for planar ${\cal C}^{1}$-paths. Let $\alpha:[0,L]\rightarrow\mathbb{R}$
be a given continuous function.

We first consider the case when the initial condition of $2\phi_{t}^{\lambda}$
coincides with $\alpha_{0}.$
\begin{lem}
\label{lem:SamInt}For each $\lambda>0,$ let $(\phi_{t}^{\lambda})_{0\leqslant t\leqslant L}$
be the solution to the differential equation
\begin{equation}
\begin{cases}
d\phi_{t}^{\lambda}=\lambda\sin(\alpha_{t}-2\phi_{t}^{\lambda})dt, & 0\leqslant t\leqslant L,\\
\phi_{0}^{\lambda}=\alpha_{0}/2.
\end{cases}\label{eq:DEAng}
\end{equation}
Then $2\phi_{t}^{\lambda}$ converges uniformly to $\alpha_{t}$ on
$[0,L]$ as $\lambda\rightarrow\infty.$
\end{lem}
\begin{proof}
Let 
\[
\omega(h)\triangleq\sup_{\substack{|s-t|<h}
}|\alpha_{t}-\alpha_{s}|
\]
be the modulus of continuity of $\alpha.$ Given $\varepsilon\in(0,\pi/4),$
choose $h=h_{\varepsilon}$ so that $\omega(h)<\varepsilon.$ Let
$\lambda>\frac{\|\alpha\|_{\infty}}{h\sin\varepsilon}.$ We claim
that $|2\phi_{t}^{\lambda}-\alpha_{t}|<2\varepsilon$ for all $t\in[0,L].$
Suppose on the contrary that $|2\phi_{t}^{\lambda}-\alpha_{t}|\geqslant2\varepsilon$
for some $t$. Define 
\[
t_{2}\triangleq\inf\{s\in[0,L]:|2\phi_{s}^{\lambda}-\alpha_{s}|\geqslant2\varepsilon\},
\]
and
\[
t_{1}\triangleq\sup\{0\leqslant s\leqslant t_{2}:|2\phi_{s}^{\lambda}-\alpha_{s}|\leqslant\varepsilon\}.
\]
Apparently, $0<t_{1}<t_{2}\leqslant L.$ Moreover, $|2\phi_{t_{1}}^{\lambda}-\alpha_{t_{1}}|=\varepsilon$
and 
\[
\varepsilon\leqslant|2\phi_{s}^{\lambda}-\alpha_{s}|\leqslant2\varepsilon\ \ \ \forall s\in[t_{1},t_{2}].
\]
Using the differential equation (\ref{eq:DEAng}), we also have 
\begin{equation}
2\phi_{t_{2}}^{\lambda}-\alpha_{t_{2}}=2\phi_{t_{1}}^{\lambda}-\alpha_{t_{1}}+2\lambda\int_{t_{1}}^{t_{2}}\sin(\alpha_{s}-2\phi_{s}^{\lambda})ds+(\alpha_{t_{1}}-\alpha_{t_{2}}).\label{eq: ODE estimate sin}
\end{equation}

If $2\phi_{t_{2}}^{\lambda}-\alpha_{t_{2}}=2\varepsilon,$ by the
definition of $t_{1}$ we must have $2\phi_{t_{1}}^{\lambda}-\alpha_{t_{1}}=\varepsilon$
and thus
\[
\varepsilon\leqslant2\phi_{s}^{\lambda}-\alpha_{s}\leqslant2\varepsilon\ \ \ \forall s\in[t_{1},t_{2}].
\]
Therefore, (\ref{eq: ODE estimate sin}) implies that
\begin{align*}
2\varepsilon & \leqslant\varepsilon-2\lambda(t_{2}-t_{1})\sin\varepsilon+|\alpha_{t_{1}}-\alpha_{t_{2}}|\\
 & \leqslant\begin{cases}
\varepsilon+\omega(h)<2\varepsilon, & {\rm if}\ t_{2}-t_{1}<h;\\
\varepsilon+2\|\alpha\|_{\infty}-2\lambda h\sin\varepsilon<\varepsilon, & {\rm if}\ t_{2}-t_{1}\geqslant h.
\end{cases}
\end{align*}
This is clearly a contradiction. The case when $2\phi_{t_{2}}^{\lambda}-\alpha_{t_{2}}=-2\varepsilon$
is treated in a similar way. Consequently, we conclude that $|2\phi_{t}^{\lambda}-\alpha_{t}|<2\varepsilon$
for all $t\in[0,L],$ provided that $\lambda>\frac{\|\alpha\|_{\infty}}{h\sin\varepsilon}$.
\end{proof}
Now we extend the above argument to give a proof of Lemma \ref{lem:C1Lem}.
Namely, under the assumption that 
\begin{equation}
|2\phi_{0}^{\lambda}-\alpha_{0}|\leqslant\kappa<\pi\ \ \ \forall\lambda>0,\label{eq:C1Initial}
\end{equation}
we want to establish the uniform convergence of $2\phi_{t}^{\lambda}$
towards $\alpha_{t}$ on $[t_{0},L]$ where $t_{0}>0$ is a given
fixed time. In order to use the previous proof, the crucial point
is to see that, when $\lambda$ is large the quantity $2\phi_{t}^{\lambda}-\alpha_{t}$
can be brought down to the region $(-\varepsilon,\varepsilon)$ at
some time in $[0,t_{0}].$ The heuristic reason for such a property
is simple to describe. Since $|2\phi_{0}^{\lambda}-\alpha_{0}|$ is
uniformly less than $\pi$, at the initial stage of the dynamics (i.e.
when $t$ is small), the quantity $\sin(2\phi_{t}^{\lambda}-\alpha_{t})$
is uniformly away from zero. Therefore, when $\lambda$ is large,
the mean-reversing property gets rather significant and is thus pushing
$2\phi_{t}^{\lambda}$ to be close to $\alpha_{t}$ very quickly.
Let us now make the heuristics precise.
\begin{lem}
\label{lem:SmallAng}Suppose that (\ref{eq:C1Initial}) holds for
some given constant $\kappa\in(0,\pi).$ Let $t_{0}>0$ be fixed.
Then for any $\varepsilon>0,$ there exists $\Lambda=\Lambda_{\varepsilon,t_{0}}>0,$
such that for each $\lambda>\Lambda$ we have
\begin{equation}
|2\phi_{s}^{\lambda}-\alpha_{s}|<\varepsilon\ \ \ \text{for some }s\in[0,t_{0}].\label{eq:SmallAng}
\end{equation}
\end{lem}
\begin{proof}
Let $\kappa'\in(\kappa,\pi)$ be fixed. By the continuity of $\alpha_{t}$
at the origin, there exists $\delta\in(0,t_{0})$ such that 
\[
t\in[0,\delta]\implies|\alpha_{t}-\alpha_{0}|<\kappa'-\kappa.
\]
Given $\varepsilon>0,$ we define 
\[
\Lambda\triangleq\frac{2\|\alpha\|_{\infty}+\kappa}{2\delta\sin\varepsilon}.
\]
For each given $\lambda>\Lambda,$ we claim that (\ref{eq:SmallAng})
holds. Suppose on the contrary that 
\[
|2\phi_{s}^{\lambda}-\alpha_{s}|\geqslant\varepsilon\ \ \ \forall s\in[0,t_{0}].
\]
By continuity, we have either 
\[
(\text{i})\ 2\phi_{s}^{\lambda}-\alpha_{s}\geqslant\varepsilon\ \ \ \forall s\in[0,t_{0}]
\]
or
\[
(\text{ii})\ 2\phi_{s}^{\lambda}-\alpha_{s}\leqslant-\varepsilon\ \ \ \forall s\in[0,t_{0}].
\]
Suppose that Case (i) holds. Define 
\[
s_{1}\triangleq\inf\{s\in[0,t_{0}]:2\phi_{s}^{\lambda}-\alpha_{s}=\kappa'\}.
\]
Then we must have $s_{1}>\delta.$ Indeed, consider the equation 
\begin{equation}
2\phi_{s_{1}}^{\lambda}-\alpha_{s_{1}}=2\phi_{0}^{\lambda}-\alpha_{0}-\big(\alpha_{s_{1}}-\alpha_{0}\big)-2\lambda\int_{0}^{s_{1}}\sin(2\phi_{s}^{\lambda}-\alpha_{s})ds.\label{eq:Ints_1}
\end{equation}
Note that 
\[
\varepsilon\leqslant2\psi_{s}^{\lambda}-\alpha_{s}\leqslant\kappa'<\pi\ \ \ \forall s\in[0,s_{1}]
\]
and thus $\sin(2\phi_{s}^{\lambda}-\alpha_{s})$ is positive on $[0,s_{1}].$
If $s_{1}\leqslant\delta,$ the left hand side of (\ref{eq:Ints_1})
equals $\kappa'$ while the right hand side is strictly less than
\[
\kappa+(\kappa'-\kappa)-2\lambda\int_{0}^{s_{1}}\sin(2\phi_{u}^{\lambda}-\alpha_{u})du\leqslant\kappa'.
\]
This is clearly a contradiction. Therefore, $s_{1}>\delta.$ Now using
the same equation (\ref{eq:Ints_1}), we see that the left hand side
is bounded below by $\varepsilon$ while the right hand side is bounded
above by 
\[
\kappa+2\|\alpha\|_{\infty}-2\lambda s_{1}\sin\varepsilon\leqslant\kappa+2\|\alpha\|_{\infty}-2\lambda\delta\sin\varepsilon.
\]
This leads to a contradiction, since the quantity on the right hand
side of the above inequality is negative when $\lambda>\Lambda$ by
the definition of $\Lambda.$ The discussion of Case (ii) is similar.
\end{proof}
Now we are able to complete the proof of Lemma \ref{lem:C1Lem}.

\begin{proof}[Proof of Lemma \ref{lem:C1Lem}]

Given $\varepsilon>0,$ define
\[
\Lambda_{1}\triangleq\frac{\|\alpha\|_{\infty}}{h\sin\varepsilon},\ \Lambda_{2}\triangleq\frac{2\|\alpha\|_{\infty}+\kappa}{2\delta\sin\varepsilon},
\]
which are the two constants appearing in the proofs of Lemma \ref{lem:SamInt}
and Lemma \ref{lem:SmallAng} respectively. Define $\Lambda\triangleq\max\{\Lambda_{1},\Lambda_{2}\}$.
For each $\lambda>\Lambda,$ we know from Lemma \ref{lem:SmallAng}
that there is $s\in[0,t_{0}]$ (which may depend on $\lambda$) such
that (\ref{eq:SmallAng}) holds. In addition, exactly the same argument
as in the proof of Lemma \ref{lem:SamInt} allows us to conclude that
\[
|2\phi_{t}^{\lambda}-\alpha_{t}|\leqslant2\varepsilon\ \ \ \forall t\in[s,L]
\]
In particular, we have
\[
\sup_{t\in[t_{0},L]}|2\phi_{t}^{\lambda}-\alpha_{t}|\leqslant2\varepsilon.
\]
This gives the desired uniform convergence.

\end{proof}
\begin{rem}
It is not hard to see why $2\phi_{0}^{\lambda}=\alpha_{0}\pm\pi$
are ``bad'' initial condition. Consider the extreme example where
$\alpha_{t}\equiv\alpha_{0}.$ If $2\phi_{0}^{\lambda}=\alpha_{0}\pm\pi$,
then we have $\phi_{t}^{\lambda}\equiv\phi_{0}^{\lambda}$, which
is never close to $\frac{\alpha_{t}}{2}$. If we perform an explicit
calculation for the tree-like path $v\sqcup(-v)$ ($v\in\mathbb{R}^{2}$),
this is exactly what happens in the $(-v)$-part.
\end{rem}
\begin{cor}
Let $\gamma:[0,L]\rightarrow\mathbb{R}^{2}$ be a path defined by
the equation (\ref{eq:DefPath}), where the angular path $\beta:[0,L]\rightarrow\mathbb{R}^{2}$
is a continuous function. Then the signature asymptotics formula (\ref{eq:LengConj})
holds for $\gamma.$
\end{cor}
\begin{proof}
With $\alpha_{t}\triangleq\beta_{L-t}$ and $2\phi_{0}^{\lambda}\triangleq\alpha_{0},$
Lemma \ref{lem:C1Lem} shows that the angular path $2\phi_{t}^{\lambda}$
converges uniformly to $\alpha_{t}$ as $\lambda\rightarrow\infty.$
The result then follows from the lower estimate (\ref{eq:IntegralLow}).
\end{proof}
\end{appendices}

\begin{appendices}
\section{Proof that every path is reparametrisable to have unit speed}

Let $|\cdot|$ denote the Euclidean norm. Define 
\[
\Vert f\Vert_{1\text{-var},[s,t]}=\sup_{\mathcal{P}}\sum_{i=0}^{n-1}|f_{t_{i+1}}-f_{t_{i}}|,
\]
where the supremum is taken over all partitions $s=t_{0}<t_{1}<\ldots<t_{n}=t.$ 

We will use the following three facts about $1$-variation.

(i) For all $s\leqslant t$, we have
\[
\Vert f\Vert_{1\text{-var},[0,t]}=\Vert f\Vert_{1\text{-var},[0,s]}+\Vert f\Vert_{1\text{-var},[s,t]}.
\]

(ii) For all non-decreasing function $L:[0,1]\rightarrow[0,L(1)]$
such that $g\circ L(t)=f(t)$ with some function $g$, we have
\[
\Vert g\Vert_{1\text{-var},[0,L(s)]}=\Vert f\Vert_{1\text{-var},[0,s]}.
\]

(iii) For all $s\leqslant t$, we have
\[
|f(t)-f(s)|\leqslant\Vert f\Vert_{1\text{-var},[s,t]}.
\]

Fact (i) appeared in \cite{KF75} while Facts (ii) and (iii) appeared
without proofs in \cite[Lemma 1.6]{LCL07}.
\begin{lem}
Let $\gamma:[0,1]\rightarrow\mathbb{R}^{d}$ be a continuous path
with finite variation. Define $L(t)\triangleq\Vert\gamma\Vert_{1\text{-var},[0,t]}$.
There exists a function $\tilde{\gamma}:[0,L(1)]\rightarrow\mathbb{R}^{d}$
such that $\tilde{\gamma}_{L(t)}=\gamma_{t}$ for all $t\in[0,1]$. 
\end{lem}
\begin{proof}
If $L(t)=L(s)$, then by Fact (i) $\Vert\gamma\Vert_{1\text{-var},[s,t]}=0$
and by Fact (iii) $\gamma_{t}=\gamma_{s}$. Since $L(t)=L(s)\implies\gamma_{t}=\gamma_{s}$,
the lemma thus follows.
\end{proof}
\begin{lem}
For all $0\leqslant u\leqslant v\leqslant L(1)$, we have
\[
|\tilde{\gamma}_{v}-\tilde{\gamma}_{u}|\leqslant|v-u|.
\]
\end{lem}
\begin{proof}
Let $s\leqslant t$ be such that $L(s)=u$ and $L(t)=v$. Note that
\begin{align*}
|\tilde{\gamma}_{v}-\tilde{\gamma}_{u}| & \leqslant\Vert\tilde{\gamma}\Vert_{1\text{-var},[u,v]}\qquad(\text{by Fact }\text{(iii)})\\
 & =\Vert\tilde{\gamma}\Vert_{1\text{-var},[0,v]}-\Vert\tilde{\gamma}\Vert_{1\text{-var},[0,u]}\quad(\text{by Fact }\text{(i)})\\
 & =\Vert\tilde{\gamma}\Vert_{1\text{-var},[0,L(t)]}-\Vert\tilde{\gamma}\Vert_{1\text{-var},[0,L(s)]}\\
 & =\Vert\gamma\Vert_{1\text{-var},[0,t]}-\Vert\gamma\Vert_{1\text{-var},[0,s]}\quad(\text{by Fact \text{(ii)})}\\
 & =v-u.
\end{align*}
The result thus follows.
\end{proof}
\begin{lem}
\label{lem:Representation}There exists $g:[0,L]\rightarrow\mathbb{R}^{d}$
such that $g\in L^{1}$, 
\[
\tilde{\gamma}_{v}-\tilde{\gamma}_{0}=\int_{0}^{v}g(\alpha)\mathrm{d}\alpha
\]
and $|g(\alpha)|=1$ for almost all $\alpha$.
\end{lem}
\begin{proof}
Since $\tilde{\gamma}$ is Lipschitz, there exists a $L^{1}$-function
$g:[0,L]\rightarrow\mathbb{R}^{d}$ (see e.g. \cite{KF75}) such that
for all $v\in[0,L(1)]$, we have
\[
\tilde{\gamma}_{v}-\tilde{\gamma}_{0}=\int_{0}^{v}g(\alpha)\mathrm{d}\alpha.
\]
By Lebesgue's differentiation theorem (again see e.g. \cite{KF75}),
for almost all $\alpha$ we have
\begin{align*}
|g(\alpha)| & =\lim_{\varepsilon\rightarrow0}\big|\frac{\tilde{\gamma}_{\alpha+\varepsilon}-\tilde{\gamma}_{\alpha}}{\varepsilon}\big|\leqslant1\qquad(\text{by Lemma \ref{lem:Representation}}).
\end{align*}
Now assume on the contrary that 
\[
\mu(\{\alpha:|g(\alpha)|<1\})>0.
\]
Then there exists $\varepsilon>0$ such that 
\[
\mu(\{\alpha:|g(\alpha)|<1-\varepsilon\})>0.
\]
It follows that
\begin{align*}
L(1)= & \Vert\gamma\Vert_{1\text{-var},[0,1]}=\Vert\tilde{\gamma}\Vert_{1\text{-var},[0,L(1)]}\quad(\text{by Fact }\text{(ii)})\\
= & \sup_{(t_{i})}\sum_{i}|\tilde{\gamma}_{t_{i+1}}-\tilde{\gamma}_{t_{i}}|,\qquad(\text{with }(t_{i})\text{ being a partition of }[0,L(1)])\\
\leqslant & \sup_{(t_{i})}\sum_{i}\int_{t_{i}}^{t_{i+1}}|g(\alpha)|\mathrm{d}\alpha=\int_{0}^{L(1)}|g(\alpha)|\mathrm{d}\alpha\\
\leqslant & (1-\varepsilon)\mu(\{\alpha:|g(\alpha)|<1-\varepsilon\})+\mu(\{\alpha:|g(\alpha)|\geqslant1-\varepsilon\})\\
= & L(1)-\varepsilon\mu(\{\alpha:|g(\alpha)|<1-\varepsilon\}),
\end{align*}
which leads to a contradiction.
\end{proof}
\begin{thm}
Let $\gamma:[0,1]\rightarrow\mathbb{R}^{2}$ be a continuous path
with finite $1$-variation. Denote 
\[
L=\Vert\gamma\Vert_{1\text{-var},[0,1]}.
\]
Then there exists a non-decreasing function $q:[0,1]\rightarrow[0,L]$
and a measurable function $\beta:[0,L]\rightarrow\mathbb{R}$ such
that 
\[
\gamma_{t}-\gamma_{0}=\int_{0}^{t}(\cos\beta_{u},\sin\beta_{u})\mathrm{d}u.
\]
\end{thm}
\begin{proof}
This is a direct consequence of Lemma \ref{lem:Representation} because
for any $(x,y)\in\mathbb{R}^{2}$ such that $x^{2}+y^{2}=1$, there
exists $\beta\in\mathbb{R}$ such that 
\[
(\cos\beta,\sin\beta)=(x,y).
\]
If we interpret ``$\arctan$'' as a function $\arctan:\mathbb{R}\rightarrow(-\frac{\pi}{2},\frac{\pi}{2})$,
then $\beta$ can be chosen in the following manner
\[
\beta=\begin{cases}
\arctan(\frac{y}{x}), & x>0;\\
\arctan(\frac{y}{x})-\pi, & x<0,y<0;\\
\arctan(\frac{y}{x})+\pi, & x<0,y\geqslant0;\\
\frac{\pi}{2}, & x=0,y>0;\\
-\frac{\pi}{2}, & x=0,y<0.
\end{cases}
\]
In particular, $\beta$ is a measurable function of $(x,y)$. As a
consequence, $\beta_{\cdot}$ as a measurable function of $g(\cdot)$
is measurable. 
\end{proof}
\end{appendices}

\section*{Acknowledgment}

We thank the referees for their careful review of this paper.

\end{document}